\documentclass[11pt, a4paper]{amsart}

\usepackage[ansinew]{inputenc}
\usepackage[all]{xy}
\SelectTips{cm}{}
\usepackage[pagebackref,%colorlinks,linkcolor=black,citecolor=black%
  ,pdfborder={0 0 0}
  ,urlcolor=black,a4paper,hypertexnames=false]{hyperref}
\hypersetup{pdfauthor={Francesco Fournier-Facio, Clara L\"oh, Marco Moraschini}
  ,pdftitle=Bounded cohomology and binate groups}

\usepackage{amsfonts}
\usepackage{amssymb, amsmath,amsthm, mathtools}
\usepackage{tabularx, hyperref}
\usepackage{enumerate}

\usepackage{tikz-cd}

\usepackage{stmaryrd}

\makeatletter
\newtheorem*{rep@theorem}{\rep@title}
\newcommand{\newreptheorem}[2]{%
\newenvironment{rep#1}[1]{%
 \def\rep@title{#2 \ref{##1}}%
 \begin{rep@theorem}}%
 {\end{rep@theorem}}}
\makeatother

\newtheorem{lemma}{Lemma}[section]
\newtheorem{thm}[lemma]{Theorem}
\newreptheorem{thm}{Theorem}  %for repeating the theorem
\newtheorem{prop}[lemma]{Proposition}
\newreptheorem{prop}{Proposition}  %for repeating the proposition
\newtheorem{cor}[lemma]{Corollary}
\newreptheorem{cor}{Corollary}  %for repeating the corollary

\theoremstyle{definition}
\newtheorem{defi}[lemma]{Definition}

\newtheorem{quest}[lemma]{Question}
\newreptheorem{quest}{Question}  %for repeating the question

\newtheorem{example}[lemma]{Example}
\newtheorem{rem}[lemma]{Remark}

\newtheorem{claim}{Claim}

\theoremstyle{definition}

\makeatletter
\newcommand\norm{\bBigg@{0.8}}

\makeatother
 
 \newcommand{\indnorm}[2][flex]{\csname #1l\endcsname\|#2%
                                 \csname #1r\endcsname\|\mathclose{}}
                                  \newcommand{\indnorml}[4][flex]{\csname #1l\endcsname\|#2%
                                 \csname #1r\endcsname\|_{#3}^{#4}\mathclose{}}
\newcommand{\sv}[2][flex]{\indnorm[#1]{#2}}

\newcommand{\genrel}[3][flex]{\csname #1l\endcsname\langle #2 \mathbin{\csname #1m\endcsname|} #3\csname #1r\endcsname\rangle}

\def\quand{\quad\text{and}\quad}

\DeclareMathOperator{\GL}{GL}

\DeclareMathOperator{\sgn}{sgn}
\DeclareMathOperator{\eu}{eu}
\def\eub#1{%
  \eu_b^{#1}}
\DeclareMathOperator{\orc}{or}

\DeclareMathOperator{\ubc}{UBC}

\DeclareMathOperator{\comp}{comp}

\DeclareMathOperator{\im}{im}
\DeclareMathOperator{\Aut}{Aut}
\DeclareMathOperator{\id}{id}
\DeclareMathOperator{\Homeo}{Homeo}

\newcommand{\HH}{\operatorname{H}}

\newcommand{\CC}{\operatorname{C}}

\DeclareMathOperator{\alt}{alt}
\def\linf{\ell^{\infty}}
\def\linfalt{\linf_{\alt}}

\newcommand{\qand}{%
  \qquad\text{and}\qquad}

\newcommand{\N}{\ensuremath {\mathbb{N}}}
\newcommand{\R} {\ensuremath {\mathbb{R}}}
\newcommand{\C} {\ensuremath {\mathbb{C}}}
\newcommand{\Q} {\ensuremath {\mathbb{Q}}}
\newcommand{\Z} {\ensuremath {\mathbb{Z}}}

\newcommand{\Fp} {\ensuremath {\mathbb{F}_p}}
\newcommand{\K} {\ensuremath {\mathbb{K}}}

\renewcommand{\rho}{\varrho}
\def\phi{\varphi}

\def\args{\;\cdot\;}

\def\actson{\curvearrowright}

\usepackage{color}
\usepackage{pdfcolmk}
\long\def\forget#1{}

\makeatletter
\@namedef{subjclassname@2020}{%
  \textup{2020} Mathematics Subject Classification}
\makeatother

\begin{document}

\title{Bounded cohomology and binate groups}

\author[]{Francesco Fournier-Facio}
\address{Department of Mathematics, ETH Z\"urich, Z\"urich, Switzerland}
\email{francesco.fournier@math.ethz.ch}

\author[]{Clara L\"oh}
\address{Fakult\"{a}t f\"{u}r Mathematik, Universit\"{a}t Regensburg, Regensburg, Germany}
\email{clara.loeh@ur.de}

\author[]{Marco Moraschini}
\address{Department of Mathematics, University of Bologna, Bologna, Italy}
\email{marco.moraschini2@unibo.it}

\thanks{}

%18G90          Other co(ho)mological theories
\date{\today.\ 
  Clara L\"oh and Marco Moraschini have been supported by the CRC~1085 \emph{Higher Invariants}
  (Universit\"at Regensburg, funded by the~DFG).
  The results in this paper are part of Francesco Fournier-Facio's PhD project}
  
\maketitle

\begin{abstract}
A group is boundedly acyclic if its bounded cohomology with
trivial real coefficients vanishes in all positive degrees. Amenable
groups are boundedly acyclic, while the first non-amenable examples
were the group of compactly supported homeomorphisms of~$\R^n$
(Matsu\-moto--Morita) and mitotic groups (L\"oh). 
We prove that binate (alias pseudo-mitotic) groups are boundedly
acyclic, which provides a unifying approach to the aforementioned
results.  Moreover, we show that binate groups are universally
boundedly acyclic.

We obtain several new examples of boundedly acyclic groups as well as
computations of the bounded cohomology of certain groups acting on the
circle.  In particular, we discuss how these results suggest that the
bounded cohomology of the Thompson groups $F$, $T$, and $V$ is
as simple as possible. 
\end{abstract}

%%%%%%%%%%%%%%%%%%%%%%%%%%%%%%%%%%%%%%%%%%%%%%%%%%%%%%%%%%%
\section{Introduction}

Bounded cohomology is defined via the topological dual of the
simplicial resolution. It was introduced by Johnson and Trauber in the
context of Banach algebras~\cite{Johnson}, then extended by Gromov to
topological spaces~\cite{vbc}.  Since then it has become a fundamental
tool in several fields, including the geometry of manifolds~\cite{vbc},
rigidity theory~\cite{rigidity}, the dynamics of circle
actions~\cite{ghys}, and stable commutator length~\cite{calegari}. 

Despite a good understanding in degree $2$ and a partial understanding
in degree $3$, the full bounded cohomology of a group seems to be hard
to compute~\cite[Section~7]{our}. Therefore, it is fundamental to
produce alternative resolutions that compute the bounded cohomology of
a group.  In this respect, amenable groups played a fundamental role
in the approach of Ivanov~\cite{ivanov}. One can also exploit a larger
class of groups for computing bounded cohomology, namely the class of
\emph{boundedly acyclic groups}~\cite{BAc,Ivanov_bac_covers}:

\begin{defi}\label{defi:bac}
Let $n \geq 1$. A group~$\Gamma$ is said to be \emph{$n$-boundedly
  acyclic} if $\HH^i_b(\Gamma; \R) \cong 0$ for all~$i \in \{1,\dots,n\}$.
The group~$\Gamma$ is \emph{boundedly acyclic} if it is $n$-boundedly
acyclic for every $n \geq 1$.
\end{defi}

Amenable groups are boundedly acyclic~\cite{Johnson,vbc}. The first
non-amenable example, due to Matsumoto and Morita~\cite{MM} is 
the group~$\Homeo_c(\R^n)$ of compactly supported homeomorphisms
of~$\R^n$. Their proof relies on the \emph{acyclicity} of this group,
which for the purposes of this paper will always be intended with
respect to the integers:

\begin{defi}\label{defi:acyclic}
A group $\Gamma$ is said to be \emph{acyclic} if $\HH_n(\Gamma; \Z)
\cong 0$ for all~$n \geq 1$.
\end{defi}

It was shown by Mather that $\Homeo_c(\R^n)$ is acyclic~\cite{Mather},
and thus the proof of bounded acyclicity reduces to the proof of
injectivity of the comparison map from bounded to
ordinary cohomology. The same approach was employed by L\"oh to prove
that \emph{mitotic} groups are boundedly acyclic~\cite{Loeh}.  This
class was introduced by Baumslag, Dyer, and Heller~\cite{BDH} to
produce embedding results into finitely generated acyclic
groups. Bounded acyclicity of mitotic groups, together with
co-amenability of ascending HNN extensions, eventually led to
finitely generated and finitely presented examples of non-amenable
boundedly acyclic groups~\cite{our}.

The two bounded acyclicity results mentioned above are similar in
spirit but independent of one another, since $\Homeo_c(\R^n)$ is
\emph{not} mitotic~\cite{pseudonotmitotic}. On the other hand, there
is a larger framework that includes both $\Homeo_c(\R^n)$ and mitotic
groups: \emph{binate groups} (see Section~\ref{s:binate} for
the definition).  This class was introduced by Berrick~\cite{Berrick} 
and independently by Varadarajan~\cite{Varadarajan},
under the name \emph{pseudo-mitotic}. They proved
that binate groups are acyclic, thus providing a unified approach to
the proofs of Mather and Baumslag--Dyer--Heller, as well as several
new and interesting examples of acyclic groups, mainly among groups of
homeomorphisms (see Section~\ref{ss:examples}).
We adapt this unification to bounded cohomology:

\begin{thm}[Theorem~\ref{thm:pseudo:mit:bac}]\label{thm:binatebac}
All binate groups are boundedly acyclic.
\end{thm}

We remark that in general binate groups are non-amenable, since they
typically contain free subgroups.  However, there are a few exceptions
(Section~\ref{sss:amenable:binate}).

Binate groups reflect enough of group theory to serve as a faithful
testing class for open problems such as the Bass Conjecture, a
modified version of the Baum--Connes Conjecture or the Kervaire
Conjecture~\cite{berrick_dichotomy}. By Theorem~\ref{thm:binatebac},
boundedly acyclic groups also serve
as conjecture testers. For instance, if the Bass Conjecture holds for
all boundedly acyclic groups, then it holds for all groups. This is
especially interesting since amenable groups, which serve as the
prototypical example of boundedly acyclic groups, are known to
satisfy the Bass Conjecture~\cite{Bass}. 

The bounded acyclicity of binate groups is a phenomenon that is not
strictly linked to real coefficients. Indeed we prove:

\begin{thm}[Corollary~\ref{cor:ubac}]
\label{thm:binateubac}
Binate groups are \emph{universally boundedly acyclic}: If
$\Gamma$ is a binate group, then for every complete valued field~$\K$
and every~$n \geq 1$, we have $\HH^n_b(\Gamma; \K) \cong 0$.
\end{thm}

More generally, we characterize universally boundedly acyclic groups as
those groups that are simultaneously acyclic and boundedly acyclic
(Theorem~\ref{thm:ubac}). In this sense, Theorem \ref{thm:binateubac} is a
combination of the acyclicity result of Berrick and Varadarajan,
together with Theorem \ref{thm:binatebac}, but it also contains both
results in its statement. 

\subsection*{Hereditary properties of boundedly acyclic groups}

By analogy with the amenable case, it is interesting to check which
group-theoretic constructions preserve bounded acyclicity.  It is
known that extensions, as well as quotients with boundedly acyclic
kernels, do~\cite{BAc}.  It is therefore natural to wonder whether
these two results extend to a $2$-out-of-$3$ property for bounded
acyclicity and group extensions.

Using the fact that every group embeds $2$-step subnormally into a
binate group, we show that this cannot hold. More precisely:

\begin{thm}[Theorem \ref{thm:normal subgroups}]
There exists a boundedly acyclic group~$\Gamma$ with a normal subgroup~$H$ such that
$\Gamma/H$ is boundedly acyclic, but
$\HH^n_b(H; \R)$ is continuum-dimensional for every~$n \geq 2$.
\end{thm}

We also look at directed unions of boundedly acyclic groups, and show
that these are boundedly acyclic under an additional technical
requirement (Proposition~\ref{prop:dirun}), which is however not
needed in degree~$2$ (Corollary~\ref{cor:degree2dirun}).

\subsection*{Application to Thompson groups}

The advantage of Theorem \ref{thm:binatebac} is that the class of
binate groups is flexible enough that one can construct several
concrete examples of boundedly acyclic groups. We use this to study
the bounded cohomology of certain analogues of the classical Thompson
groups $F$, $T$ and~$V$.  The amenability question for~$F$ is one of
the most influential open questions in modern group theory. It is
therefore natural to wonder whether $F$ is at least boundedly
acyclic. It is known that $F$ is $2$-boundedly acyclic, but nothing
seems to be known in higher degrees.  Using
Theorem~\ref{thm:binatebac}, we show that a countably singular
analogue of the Thompson group~$F$ is boundedly acyclic
(Proposition~\ref{prop:OF}).

Moreover, we prove that if $F$ is $n$-boundedly acyclic, then the
bounded cohomology of~$T$ is generated by the real Euler class and its
cup powers, up to degree~$n$ (Corollary~\ref{cor:HbT}). In particular,
we obtain: 

\begin{thm}[Corollary~\ref{cor:HbT}/\ref{cor:HbTr}]\label{thm:thompson:intro}
  If the Thompson group~$F$ is boundedly acyclic, then
  $\HH^*_b(T;\R)$ (with the cup-product structure) is isomorphic
  to the polynomial ring~$\R[x]$ with~$|x| =2$ and the bounded
  Euler class of~$T$ is a polynomial generator of~$\HH^*_b(T;\R)$.
  Moreover, the canonical semi-norm on~$\HH_b^*(T;\R)$ then is a norm.
\end{thm}

Therefore, the bounded acyclicity of~$F$
would make $T$ into the first group of type~$F_\infty$ that is not
boundedly acyclic, and whose bounded cohomology ring can be completely
and explicitly be computed. Similarly, if $F$ is boundedly acyclic,
then the bounded cohomology ring of~$T^{\times r}$ is a polynomial ring
in $r$~generators of degree~$2$ (Corollary~\ref{cor:HbTr}).

Independently, Monod and Nariman recently established
analogous results for the bounded cohomology of the group of
orientation-preserving homeomorphisms of~$S^1$~\cite{monodnariman}.

\subsection*{A note from the future}

After this paper was finished, Monod proved that the Thompson
group~$F$ is boundedly acyclic~\cite{monod:thompson}.  Therefore,
Theorem~\ref{thm:thompson:intro} shows that the bounded cohomology of
Thompson group~$T$ is isomorphic to the polynomial ring~$\R[x]$
with~$|x| =2$ and that the bounded Euler class of~$T$ is a polynomial
generator of~$\HH^*_b(T;\R)$.

\subsection*{Organisation of this article}

We recall the definition of bounded cohomology and the uniform
boundary condition in Section~\ref{s:bc}. Binate groups are surveyed
in Section~\ref{s:binate}. In Section~\ref{s:hereditary}, we study
hereditary properties of boundedly acyclic groups.
Section~\ref{s:ubac} is devoted to universal bounded acyclicity.  The
applications to Thompson groups are discussed in
Section~\ref{s:thompson}. Finally,
Appendix~\ref{appendix:proof:pseudo} contains the proof of
Theorem~\ref{thm:binatebac}. 

\subsection*{Acknowledgements}

We wish to thank Jon Berrick, Jonathan Bowden, Amir Khodayan Karim,
Kevin Li, Yash Lodha, Antonio L\'{o}pez Neumann, Nicolas Monod, Sam
Nariman and George Raptis for helpful discussions.

%%%%%%%%%%%%%%%%%%%%%%%%%%%%%%%%%%%%%%%%%%%%%%%%%%%%%%%%%%%
\section{Bounded cohomology}\label{s:bc}

We quickly recall basic notions concerning bounded cohomology.

%%%%%%%%%%%%%%%%%%
\subsection{Definition of bounded cohomology}\label{ss:bc}

Let $\Gamma$ be a group and let $\R \to \linf(\Gamma^{*+1})$ be the
bounded simplicial $\Gamma$-resolution of~$\R$.  More generally, if
$V$ is a normed $\Gamma$-module, we consider the
complex~$\linf(\Gamma^{*+1},V)$ and set
\[ \CC_b^*(\Gamma;V) := \linf(\Gamma^{*+1},V)^\Gamma.
\]
The \emph{bounded cohomology} of~$\Gamma$ with coefficients in~$V$
is defined as
\[ \HH_b^*(\Gamma;V) := \HH^* \bigl( \CC_b^*(\Gamma;V)\bigr).  
\]
The norm on~$\CC_b^*(\Gamma;V)$ induces a semi-norm on~$\HH_b^*(\Gamma;V)$,
the so-called \emph{canonical semi-norm}.

The canonical inclusion~$\CC_b^*(\Gamma;V) \hookrightarrow \CC^*(\Gamma;V)$
induces a natural transformation between bounded cohomology and
ordinary cohomology, the \emph{comparison map}
\[ \comp_{\Gamma,V}^* \colon \HH_b^*(\Gamma;V) \to \HH^*(\Gamma;V).
\]
Further information on the bounded cohomology of groups (and
spaces) can be found in the literature~\cite{vbc,ivanov,Frigerio}.

In Section~\ref{s:ubac}, we will also be dealing with bounded
cohomology over different valued fields.  Recall that an
\emph{absolute value} on a field $\K$ is a multiplicative map $|\cdot|
: \K \to \R$ such that $|x| = 0$ if and only if $x = 0$; and the
triangle inequality holds: $|x + y| \leq |x| + |y|$.  We say that $\K$
is a \emph{complete valued field} if the metric induced by the
absolute value is complete.  One can then define bounded cohomology
over $\K$ in exactly the same way.

If the strong triangle inequality $|x + y| \leq \max \{ |x|, |y| \}$ holds
for all~$x,y \in \K$, then $\K$ is said to be \emph{non-Archimedean}.
Concerning the bounded cohomology over non-Archimedean fields~\cite{nonarch}, 
we will only use the following result:

\begin{lemma}[{\cite[Corollary 9.38]{nonarch}}]
\label{lem:nonarch}
Let $\Gamma$ be a group, let $n \geq 1$, and suppose that
$\HH_{n-1}(\Gamma; \Z)$ is finitely generated.  Then the comparison
map
$$\comp_{\Gamma,\K}^n : \HH^n_b(\Gamma; \K) \to \HH^n(\Gamma; \K)$$
is injective.
\end{lemma}

%%%%%%%%%%%%%%%%%%
\subsection{The uniform boundary condition}
\label{ss:ubc}

We recall the uniform boundary condition, originally due to Matsumoto
and Morita~\cite{MM}, and some of its variations~\cite{Loeh}.

\begin{defi}[Uniform boundary condition]\label{defi:ubc}
Let $n \in \N$ and let $\kappa \in \R_{> 0}$. A group $\Gamma$
satisfies the $(n, \kappa)$-\emph{uniform boundary condition}, or
simply $(n, \kappa)$-$\ubc$, if for every $z \in \im \partial_{n+1}
\subset \CC_n(\Gamma; \R)$ there exists a chain~$c \in
\CC_{n+1}(\Gamma; \R)$ with
$$
\partial_{n+1} c = z \quand \|c\|_1 \leq \kappa \cdot \|z\|_1.
$$
A group~$\Gamma$ satisfies $n$-$\ubc$ if it has $(n, \kappa)$-$\ubc$ for some
$\kappa \in \R_{> 0}$.
\end{defi}

The uniform boundary condition can lead to bounded acyclicity:

\begin{thm}[{\cite[Theorem~2.8]{MM}}]\label{thm:MM:ubc}
Let $\Gamma$ be a group and let $n \in \N$. Then, the following are
equivalent:
\begin{enumerate}
\item The group $\Gamma$ satisfies $n$-$\ubc$;

\item The comparison map $\comp_{\Gamma,\R}^{n+1} \colon
  \HH_b^{n+1}(\Gamma; \R) \to \HH^{n+1}(\Gamma; \R)$ is injective.
\end{enumerate}
In particular: Every acyclic group that satisfies~$\ubc$ in all
positive degrees is boundedly acyclic.
\end{thm}

In the proof of Theorem~\ref{thm:binatebac}, it will be useful to
extend the definition of $\ubc$ from groups to group
homomorphisms~\cite[Definition~4.5]{Loeh}.

\begin{defi}[$\ubc$ for homomorphisms]\label{defi:ubc:homo}
Let $n \in \N$ and let $\kappa \in \R_{> 0}$. A group homomorphism
$\varphi \colon H \to \Gamma$ satisfies the $(n, \kappa)$-uniform
boundary condition, or simply $(n, \kappa)$-$\ubc$, if there exits a
linear map
$$
S \colon \partial_{n+1}(\CC_{n+1}(H; \R)) \to \CC_{n+1}(\Gamma; \R)
$$
with 
$$
\partial_{n+1} \circ S = \CC_n(\varphi; \R) \quand \sv{S} \leq \kappa.
$$
Here $\sv{S}$ is the operator norm of~$S$ with respect to the $\ell^1$-norms.
\end{defi}

The uniform boundary condition will be more systematically
reviewed in a forthcoming paper~\cite{liloehmoraschini}.

%%%%%%%%%%%%%%%%%%%%%%%%%%%%%%%%%%%%%%%%%%%%%%%%%%%%%%%%%%%%%%%%%
\section{Binate groups (alias pseudo-mitotic groups)}
\label{s:binate}

We recall basic notions, properties, and examples of binate (alias
pseudo-mitotic) groups. We begin with the original definition given by
Berrick~\cite{Berrick}:

\begin{defi}[Binate]\label{defi:binate}
Let $\Gamma$ be a group. We say that $\Gamma$ is \emph{binate} if for
every finitely generated subgroup $H \leq \Gamma$ there exists a
homomorphism~$\phi \colon H \to \Gamma$ and an element $g \in \Gamma$
such that for every $h \in H$, we have $$h = [g, \phi(h)] =
g^{-1}\phi(h)^{-1}g\phi(h).$$
\end{defi}

We will rather work with the equivalent notion of pseudo-mitotic
groups, introduced by Varadajan~\cite{Varadarajan}: Here an extra
homomorphism $H \to \Gamma$ is taken as part of the structure, which
leads to more transparent proofs. We refer the reader to the
literature~\cite[Remark 2.3]{BerrickVaradarajan} for a proof of the
equivalence, and point out that the terminology \emph{binate} is more
commonly used.

\begin{defi}[Pseudo-mitosis]\label{defi:pseudo:mitosis}
Let $\Gamma$ be a a group and let $H \leq \Gamma$ be a subgroup.  We
say that $H$ has a \emph{pseudo}-\emph{mitosis} in $\Gamma$ if there
exist homomorphisms $\psi_0 \colon H \to \Gamma$, $\psi_1 \colon H
\to \Gamma$ and an element $g \in \, \Gamma$ such that
\begin{enumerate}
\item For every $h \in H$, we have $h \psi_1(h) = \psi_0(h)$;

\item For all $h, h' \in \, H$, we have $[h, \psi_1(h')] = 1$;

\item For every $h \in H$, we have $\psi_1(h) = g^{-1} \psi_0(h) g$.
\end{enumerate}
\end{defi}

Here is what the definition intuitively means. There exists a
homomorphism $\psi_1 \colon H \to \Gamma$ whose image commutes with
$H$. This induces a homomorphism
$$H \times H \to \Gamma \colon (h, h') \mapsto h \psi_1(h').$$
Precomposing it with the diagonal inclusion $h \mapsto (h, h)$, we get
a homomorphism $\psi_0 \colon h \mapsto h \psi_1(h)$.  In terms of
acyclicity the crucial condition is the third item: $\psi_0$ and
$\psi_1$ are conjugate inside $\Gamma$.

\begin{defi}[Pseudo-mitotic group]\label{defi:pseudo:mitotic:groups}
A group $\Gamma$ is said to be \emph{pseudo-mitotic} if all finitely
generated subgroups of~$\Gamma$ admit a pseudo-mitosis in~$\Gamma$.
\end{defi}

Varadarajan~\cite{Varadarajan} and Berrick~\cite{Berrick}
independently showed the following fundamental result:

\begin{thm}[{\cite[Theorem~1.7]{Varadarajan}}]\label{thm:pseudo:mitotic:are:acyclic}
All pseudo-mitotic groups are acyclic.
\end{thm}

In the present article, we prove that pseudo-mitotic groups are also
examples of boundedly acyclic groups (Theorem~\ref{thm:binatebac}):

\begin{thm}\label{thm:pseudo:mit:bac}
All pseudo-mitotic groups are boundedly acyclic. 
\end{thm}

Since the proof is rather technical and it closely follows the ones of
Matsumoto--Morita~\cite{MM} and L\"oh~\cite{Loeh}, we postpone it to
Appendix~\ref{appendix:proof:pseudo}.

\begin{rem}
\label{rem:2BAc}
It is an easy consequence of Theorem~\ref{thm:pseudo:mitotic:are:acyclic}
that pseudo-mitotic groups are
$2$-boundedly acyclic, namely that if $\Gamma$ is a pseudo-mitotic
group, then $\HH^2_b(\Gamma; \R) \cong 0$.  Indeed, if $\Gamma$ is a
pseudo-mitotic group and $h \in \Gamma$, by definition there exist 
homomorphisms~$\psi_0, \psi_1 \colon \langle h \rangle \to \Gamma$ and
an element~$g \in \Gamma$ such that
$$h = \psi_0(h) \psi_1(h)^{-1} = \psi_0(h) g^{-1} \psi_0(h)^{-1} g =
[\psi_0(h)^{-1}, g].$$
In fact, this commutator expression is the one appearing in the
definition of binate groups (Definition~\ref{defi:binate}).  Hence,
every element in a pseudo-mitotic group is a commutator and so the
second comparison map is injective~\cite{bavard}.  This shows that
$\HH^2_b(\Gamma; \R)$ embeds into~$\HH^2(\Gamma; \R)$, which vanishes
by Theorem~\ref{thm:pseudo:mitotic:are:acyclic}.
\end{rem}

%%%%%%%%%%%%%
\subsection{Examples}
\label{ss:examples}

We present several examples of pseudo-mitotic
groups. A more detailed discussion of these examples
can be found in Berrick's work~\cite{BerrickTop}.

We start with a combinatorial construction of pseudo-mitotic groups
containing a given group.

\begin{example}[Binate tower]
\label{ex:tower}
Let $H$ be a group. Set $H_0 := H$, and construct~$H_{i+1}$
inductively by performing an HNN-extension of~$H_i \times H_i$ so that
the embedding~$H_i \to H_{i+1}$ is a pseudo-mitosis. More precisely,
if
$$H_{i+1} := \langle H_i \times H_i; g_{i+1} \mid g_{i+1}^{-1}(h, h)
g_{i+1} = (1, h) : h \in H_i \rangle,$$
then $h \mapsto (h, 1)$ is a pseudo-mitotic embedding of~$H_i$ in~$H_{i+1}$.

By construction, the direct limit of the $H_i$ is pseudo-mitotic. It
is the initial object in a category of pseudo-mitotic groups
containing $H$~\cite{Berrick}.
\end{example}

This example shows that every group embeds into a pseudo-mitotic
group. We will see in the next section that a less canonical construction leads to
embeddings with more special properties (Proposition~\ref{prop:2step}).

The following allows to construct new binate groups from old ones:

\begin{example}
Let $(\Gamma_i)_{i \in I}$ be a family of binate groups. Then their
direct product $\prod\limits_{i \in I} G_i$ is binate
\cite[Proposition 1.7]{pseudonotmitotic}.
\end{example}

We will soon see that $\Homeo_c(\R^n)$ is binate. Therefore the
previous example shows that $\Homeo_c(\R^n)^{\N}$ is binate, whence
boundedly acyclic. A direct proof of bounded acyclicity is given by
Monod and Nariman~\cite{monodnariman}.

For comparison, note that an arbitrary direct product of amenable
groups need not be amenable.  For instance, if $\Gamma$ is a
non-amenable residually finite group, such as a non-abelian free
group, then $\Gamma$ embeds into the direct product of its finite
quotients, which is therefore not amenable.

\subsubsection{Dissipated groups}

Let us move to more concrete examples.  Varadarajan proved that the
group $\Homeo_c(\R^n)$ of compactly supported homeomorphisms of $\R^n$
is pseudo-mitotic~\cite[Theorem 2.2]{Varadarajan}. Following
Berrick~\cite{BerrickTop}, we show here that this is just an instance
of the behaviour of a larger class of groups: Dissipated boundedly
supported transformation groups.

\begin{defi}[Boundedly supported group]\label{def:bddsupport}
Let $\Gamma$ be a group acting faithfully on a set $X$, which is
expressed as a directed union of subsets~$(X_i)_{i \in I}$.
For each~$i$, let $\Gamma_i := \{ g \in \Gamma \mid g \text{ is supported on } X_i
\}$. We say that $\Gamma$ is \emph{boundedly supported} if $\Gamma$ is the
directed union of the~$\Gamma_i$.
\end{defi}

The key property that makes certain boundedly supported groups
pseudo-mitotic is the following:

\begin{defi}[Dissipators]\label{defi:dissipators}
Let $\Gamma \actson X$ and $(X_i, \Gamma_i)_{i \in I}$ be as in
Definition~\ref{def:bddsupport}.  Let $i \in I$. A \emph{dissipator}
for~$\Gamma_i$ is an element~$\rho_i \in \Gamma$ such that
\begin{enumerate}
\item $\rho_i^k(X_i) \cap X_i = \emptyset$ for all $k \geq 1$.
\item For all $g \in \Gamma_i$, the bijection of $X$ defined by
\begin{equation}\label{eq:dissipated}\tag{$*$}\phi_i(g) := 
\begin{cases}
\rho_i^k g \rho_i^{-k} & \text{on } \rho^k(X_i), \mbox{ for every } k \geq 1; \\
\id & \text{elsewhere}
\end{cases}
\end{equation}
is in $\Gamma$. 
\end{enumerate}
If for each~$i \in I$ there exists a dissipator
for~$\Gamma_i$, we say that $\Gamma$ is \emph{dissipated}.
\end{defi}

In order for $\rho_i$ to be a dissipator, the element $\phi_i(g)$
needs to belong to~$\Gamma$, and the boundedly supported hypothesis
implies that there exists~$j \in I$ such that $\rho_i^k(X_i) \subset
X_j$ for all~$k \geq 1$.  Figure~\ref{fig:dissipator} illustrates this
situation.

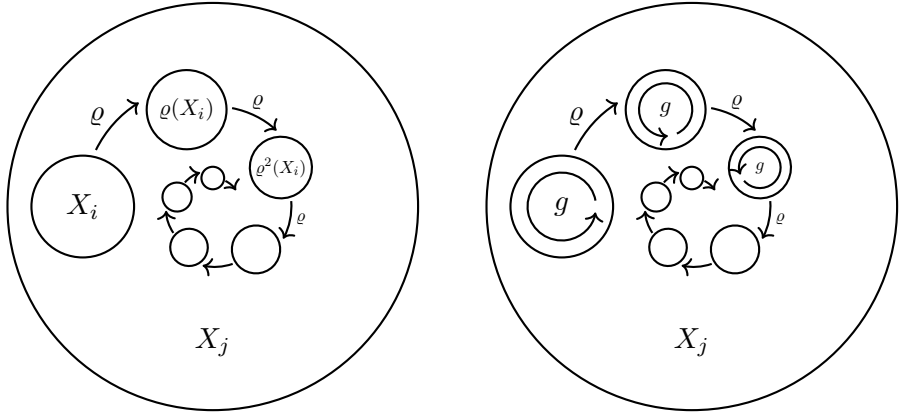
\begin{figure}
  \def\scaleXi{1.4/1.8}
  \def\rotXi{-75}
  \def\radiusXi{0.75}
  \def\codeXi{}
  \def\rhoXi{\rho}
  \def\labelXi{X_i}
  \def\pscaleXi{1}
  \def\protXi{0}
  \def\subsetXi{%
    \begin{scope}[scale={\pscaleXi},rotate={\protXi}]
      \draw (-1.9,0) circle (\radiusXi);
      \draw[->] (154:1.9) arc (154:122:1.7);
      \draw (142:2.15) node[scale={\pscaleXi}] {$\rhoXi$};
      \draw (-1.9,0) node[scale={\pscaleXi}] {$\labelXi$};
      \begin{scope}[shift={(-1.9,0)}]
        \codeXi%
      \end{scope}
    \end{scope}
    \edef\protXi{\rotXi+\protXi}
    \edef\pscaleXi{\scaleXi*\pscaleXi}
  }
  \begin{center}
    \begin{tikzpicture}[x=0.9cm,y=0.9cm,thick]
      % \rho
      \draw (0,0) circle (3);
        \def\pscaleXi{1}
        \def\protXi{0}
        \def\codeXi{}
        \def\rhoXi{\rho}
        \def\labelXi{X_i}
        \subsetXi
        \def\labelXi{\rho(X_i)}
        \subsetXi
        \def\labelXi{\rho^2(X_i)}
        \subsetXi
        \def\labelXi{}
        \def\rhoXi{}
        \subsetXi
        \subsetXi
        \subsetXi
        \subsetXi

        \draw (0,-2) node {$X_j$};
        
      %%%%%%%%%%%%%%%%%%%%%%%%%
      % \phi(g)
      \begin{scope}[shift={(7,0)}]
        \draw (0,0) circle (3);

        \def\pscaleXi{1}
        \def\protXi{0}
        \def\codeXi{%
          \draw[->] (10:0.5) arc (10:350:0.5);
          \draw (0,0) node[scale={\pscaleXi}] {$g$};
        }
        \def\rhoXi{\rho}
        \def\labelXi{}
        \subsetXi
        \subsetXi
        \subsetXi
        \def\codeXi{}
        \def\rhoXi{}
        \subsetXi
        \subsetXi
        \subsetXi
        \subsetXi

        \draw (0,-2) node {$X_j$};
        
      \end{scope}
    \end{tikzpicture}
  \end{center}

  \caption{dissipation, schematically;\newline
    left: the subsets~$\rho^k(X_i)$;
  right: the action of~$\phi(g)$.}
  \label{fig:dissipator}
\end{figure}

The presence of dissipators is enough to ensure that the group is
pseudo-mitotic:

\begin{prop}[{\cite[Section 3.1.6]{BerrickTop}}]
\label{prop:dissipated}
Dissipated groups are pseudo-mitotic.
\end{prop}

\begin{proof}
Let $\Gamma \actson X$ and $(X_i, \Gamma_i, \rho_i)_{i \in I}$ be as in the
definition of a dissipated group (Definition~\ref{defi:dissipators}).
Let $H \leq \Gamma$ be a finitely
generated subgroup. Since $\Gamma$ is boundedly supported, there
exists an~$i \in I$ such that $H \leq \Gamma_i$. Notice that $H$ commutes
with $\phi_i(H)$ (as defined in Equation~\eqref{eq:dissipated}) since
their supports are disjoint in $X$. Hence, if we define $\psi_1 \colon H
\to \Gamma$ as $\psi_1(h) := \phi_i(h)$, it is immediate to check that
$\psi_1$ is a homomorphism and that $[h', \psi_1(h)] = 1$ for all $h,
h' \in \, H$. We then set $\psi_0 := \rho_i^{-1} \psi_1 \rho_i \colon
H \to \Gamma$ and $g := \rho_i^{-1} \in \Gamma$. By construction, this
implies that $h \psi_1(h) = \psi_0(h)$ for all~$h \in \, H$. Hence,
$\psi_0, \psi_1$ and $g$ are the witnesses of a pseudo-mitosis of~$H$
in~$\Gamma$.
\end{proof}

A more topological version of this criterion is described by Sankaran
and Varadarajan~\cite[Theorem 1.5]{VaradarajanHomeo}.  Many boundedly
supported groups are dissipated, and quite surprisingly this is
usually easy to check.  We list some examples for which dissipators
can be computed directly.  More details and further constructions can
be found in Berrick's paper~\cite[Section 3.1.6]{BerrickTop} and the
references therein, as well as in the one of
Sankaran--Varadarajan~\cite{VaradarajanHomeo}.

\begin{example}[Dissipated groups]\label{ex:dissipators}
The following groups are dissipated:
\begin{enumerate}
\item The group~$\Homeo_c(\R^n)$ of compactly supported homeomorphisms of~$\R^n$
  is dissipated.  This is already contained
  in a paper of Schreier and Ulam~\cite{dissipators}, where they
  study this phenomenon for the (isomorphic) group of homeomorphisms
  of the~$n$-ball in~$\R^n$ fixing a neighbourhood of the boundary.
  Acyclicity was shown by Mather~\cite{Mather}, and the proof
  serves as a model for the proof of acyclicity of pseudo-mitotic
  groups~\cite{Varadarajan}.
\item The previous example generalizes to certain groups of boundedly
  supported homeomorphisms of topological manifolds~\cite{dissipatedmfd}
  and $C^1$-manifolds~\cite{dissipatedc1}.
\item Let $C$ be the standard Cantor set, embedded in~$[0, 1]$. Then
  the group of homeomorphisms of~$C$ that are the identity in a
  neighbourhood of $0$ and~$1$ is dissipated~\cite[Theorem
    2.4]{VaradarajanHomeo}.
\item Let $\Q$ be endowed with the topology as subspace of $\R$.
  Then, the group of homeomorphisms of~$\Q$ having support contained
  in some interval $[a, b]$ with $a < b \in \Q$ is dissipated. The
  same holds for the space of irrational numbers~\cite[Theorem
    1.13]{VaradarajanHomeo}.
\item Forgetting the topology, denote by~$\Aut(\Q)$ the group of
  bijections of~$\Q$ whose support is contained in some interval $(a,
  b)$ with $a < b \in \Q$. Then $\Aut(\Q)$ is
  dissipated~\cite[Theorem~3.2]{pseudonotmitotic}. The same holds for
  groups of bijections of infinite sets with similar properties.
\end{enumerate}
\end{example}

%%%%%%%%%%%
\subsubsection{Flabby groups}
Another source of examples are flabby groups:

\begin{defi}[Flabby group]
A group $\Gamma$ is \emph{flabby} if there exist homomorphisms $\oplus
: \Gamma \times \Gamma \to \Gamma$ and $\tau : \Gamma \to \Gamma$ such
that the following holds: For every finitely generated subgroup $H
\leq \Gamma$ there exist $a, b, c \in \Gamma$ such that for all~$h \in H$:
\begin{enumerate}
\item $h \oplus 1 = a^{-1}ha$;
\item $1 \oplus h = b^{-1}hb$;
\item $h \oplus \tau(h) = c^{-1}\tau(h)c$.
\end{enumerate}
\end{defi}

Flabby groups were defined by Wagoner~\cite{flabby}, who proved
that they are acyclic.  In fact, the following stronger result is
true:

\begin{lemma}[{\cite[Section 3.3]{Berrick}}]
Flabby groups are pseudo-mitotic.
\end{lemma}

\begin{proof}
Let $\Gamma$ be a flabby group and $H \leq \Gamma$ a finitely
generated subgroup. Let $\oplus, \tau, a, b, c$ be as in the
definition of flabby group. We define $\psi_1(h) := a(1 \oplus
\tau(h))a^{-1}$. Then, since $1 \oplus \Gamma$ commutes with $\Gamma
\oplus 1$, we have $[h', \psi_1(h)] = 1$ for all $h, h' \in H$.  Let
$\psi_0(h) := h \psi_1(h)$ for every $h \in H$. Then:
$$
\psi_0(h) = ac^{-1} b \psi_1(h) b^{-1} c a^{-1}
$$
for every $h \in \, H$. By setting $g := a c^{-1} b$, we get the thesis.
\end{proof}

The definition of a flabby group is more restrictive than that of a
pseudo-mitotic group, since the homomorphisms $\oplus$ and $\tau$
impose some uniformity in the choices of the homomorphisms $\psi_0$
and $\psi_1$. Still, there are several examples of flabby groups in
the literature.

\begin{example}[Flabby groups]\label{ex:flabby}
The following groups are flabby:
\begin{enumerate}
\item To study the algebraic $K$ theory of a ring $R$, Wagoner~\cite{flabby}
  embeds $R$ into another ring~$CR$ called the
  \emph{cone over $R$}. Then the direct limit general linear group
  $\GL(CR)$ is shown to be flabby, whence acyclic.
\item Building on the work of Wagoner, several other examples of
  flabby groups are exhibited by de la~Harpe and McDuff~\cite{flabby2},
  and they all have the
  following flavour.  Let $V$ be an (infinite-dimensional) Hilbert
  space, and let $$V = S_0 \supset S_1 \supset \cdots \supset S_i
  \supset \cdots$$ be a chain of closed subspaces such that
  $S_i/S_{i-1}$ is isomorphic to $V$ for all $i$. Let $\GL(V)$ be the
  group of continuous linear isomorphisms of $V$, and let $\Gamma_i :=
  \{ g \in \GL(V) \mid g(S_i^\perp) = S_i^\perp \}$. Then the direct
  limit of the $\Gamma_i$'s is flabby.  The same holds if one restricts
  to unitary operators.
\item Certain groups of automorphisms of measure spaces fall into the
  framework of the previous item, and thus are flabby~\cite{flabby2}.
\end{enumerate}
\end{example}

%%%%%%%%%%
\subsubsection{Mitotic groups}

Pseudo-mitotic groups were introduced by
Varadarajan~\cite{Varadarajan} as a generalization of a more
restricted class, that of \emph{mitotic groups}, introduced by
Baumslag, Dyer and Heller~\cite{BDH}. Let us recall the definition:

\begin{defi}[Mitosis]\label{defi:mitosis}
Let $\Gamma$ be a group and let $H \leq \Gamma$ be a subgroup.  We
say that $H$ has a \emph{mitosis} in $\Gamma$ if there exist elements
$s, d \in \Gamma$ such that
\begin{enumerate}
\item[(1')] For every $h \in \, H$, we have $h \cdot s^{-1} h s = d^{-1} h d$;

\item[(2')] For all $h, h' \in \, H$, we have $[h, s^{-1} h' s] = 1$.
\end{enumerate}
\end{defi}

\begin{defi}[Mitotic]\label{defi:mitotic:groups}
A group $\Gamma$ is said to be \emph{mitotic} if all finitely
generated subgroups of~$\Gamma$ admit a mitosis in $\Gamma$.
\end{defi}

The main examples of mitotic groups are algebraically closed
groups~\cite[Theorem 4.3]{BDH}; moreover, a functorial embedding
analogous to Example \ref{ex:tower} is also possible for mitotic
groups. 

If $H$ has a mitosis in~$\Gamma$, then $H$ also has a pseudo-mitosis
in~$\Gamma$: Indeed, we can choose $\psi_0$ to be conjugation by~$d$,
$\psi_1$ to be conjugation by~$s$, and $g := s^{-1} d$. Therefore,
every mitotic group is pseudo-mitotic. On the other hand, the class of
pseudo-mitotic groups is strictly larger as proved by Sankaran and
Varadarajan~\cite{pseudonotmitotic}.  Since
Theorem~\ref{thm:binatebac} generalizes the bounded acyclicity of
mitotic groups~\cite{Loeh} to the class of pseudo-mitotic groups, let
us show an explicit example of a group that is pseudo-mitotic but not
mitotic:

\begin{lemma}
\label{lem:notmitotic}
The group $\Homeo_c(\R)$ is not mitotic.
\end{lemma}

\begin{proof}
Let $H$ be a finitely generated group acting minimally on~$(0, 1)$:
For concreteness, one could take $H$ to be the Thompson group~$F$. We
embed~$H$ inside~$\Homeo_c(\R)$ by letting $H$ act trivially on~$\R
\setminus (0, 1)$.  Let $s \in \Homeo_c(\R)$ be an element such that
$H$ commutes with~$s^{-1}Hs$. Note that $s^{-1}Hs$ is supported
on~$s^{-1} (0, 1)$ and also acts minimally. So in order to commute,
$s^{-1}(0, 1)$ must be disjoint from~$(0, 1)$.

It follows that the diagonal group~$\{ h s^{-1} h s \mid h \in H \}$ is
supported on the disconnected set~$(0, 1) \cup s^{-1}(0, 1)$, and
therefore cannot be conjugate to~$H$ by an element in~$\Homeo_c(\R)$.
This shows that $\Homeo_c(\R)$ is not mitotic.
\end{proof}

With a little more work, this kind of argument can be applied to many
of the groups from Example~\ref{ex:dissipators}. Indeed, it is known 
that $\Homeo_c(\R^n)$ is not mitotic~\cite{pseudonotmitotic}, while
it is pseudo-mitotic (Example~\ref{ex:dissipators}).

%%%%%%%%%%%
\subsubsection{Amenable examples}
\label{sss:amenable:binate}

Most of the examples that we have seen so far are non-amenable, since
they contain free subgroups.  In particular, their bounded acyclicity
does not follow from the classical result for amenable groups.
However, there are some exceptions. The first one is due to Berrick:

\begin{example}
Hall's countable
universal locally finite group~\cite{hall}
is pseudo-mitotic~\cite[Section 3.1]{Berrick}.
Being locally finite, it is amenable.
\end{example}

Locally finite groups cannot be dissipated, since by construction the
dissipators must have infinite order.  In the next example, we
construct a dissipated amenable group.

\begin{example}
  We start with $\Gamma_1 = X_1 = \Z$, with the action given by left
  translation. Of course, $\Gamma_1$ is amenable.

  Next, let $X_2$ be the disjoint union of countably many copies
  of~$X_1$ indexed by~$\Z$, which contains a distinguished copy
  of~$X_1$, indexed by~$0$.  The direct product~$\Gamma_1^{\Z}$ acts
  on~$X_2$ coordinate-wise.  Let $\rho_1$ be the bijection of~$X_2$
  shifting the copies of~$X_1$. We set $\Gamma_2$ to be the group
  generated by the direct product~$\Gamma_1^{\Z}$ and~$\rho_1$.  Note
  that $\Gamma_2$ splits as a semidirect product~$\Gamma_1^{\Z}
  \rtimes \langle \rho_1 \rangle$, so it is $2$-step-solvable.
  Moreover, given~$g \in \Gamma_1$, the element~$\varphi_1(g)$ from
  Definition~\ref{defi:dissipators} is just~$(g_n)_{n \in \Z} \in
  \Gamma_1^{\Z} \leq \Gamma_2$ with $g_n = g$ for~$n \geq 0$ and $g_n
  = 0$ otherwise.

  By induction, if $\Gamma_i$ and $X_i$ have been constructed, we
  construct~$X_{i+1}$ as the disjoint union of $\Z$-many copies
  of~$X_i$, we let $\rho_i$ be the shift, and define~$\Gamma_{i+1}$ as
  the group generated by the direct product of the $\Gamma_i$
  and~$\rho_i$.  Then $\Gamma_{i+1} \cong \Gamma_i^{\Z} \rtimes \Z$ is
  $(i+1)$-step-solvable.  Moreover, given~$g \in \Gamma_i$, the
  element~$\varphi_i(g)$ again belongs to~$\Gamma_i^{\Z} \leq
  \Gamma_{i+1}$.

  The directed union~$\Gamma$ of the~$\Gamma_i$ acts on the directed
  union~$X$ of the~$X_i$.  This action is boundedly supported by
  definition, and the~$\rho_i$ are dissipators by construction.
  Finally, $\Gamma$ is a directed union of solvable groups, so it is
  amenable.
\end{example}

\begin{rem}
  In the construction, one cannot pick any amenable group~$\Gamma_1$,
  since a direct power of amenable groups need not be amenable in
  general.  Indeed, we strongly use the fact that a direct power of an
  $i$-step-solvable group is still $i$-step-solvable.  By the same
  argument, we could have started with any group satisfying an
  amenable law.
\end{rem}

%%%%%%%%%%%%%%%%%%%%%%%%%%%%%%%%%%%%%%%%%%%%%%%%%%%%%%%%%%%%%%%%
\section{Hereditary properties of boundedly acyclic groups}\label{s:hereditary}

In this section, we discuss the stability of bounded acyclicity under
certain operations.  We will present new results concerning normal
subgroups and directed unions.  The case of normal subgroups will make
use of pseudo-mitotic groups, showcasing their versatility compared to
mitotic groups.

%%%%%%%%%%
\subsection{Extensions}

We start with the operation of taking extensions of boundedly acyclic
groups. This behaves particularly well:

\begin{thm}[{\cite[Corollary 4.2.2]{BAc}}]
\label{thm:ext}

Let $1 \to N \to \Gamma \to Q \to 1$ be an exact sequence of groups,
where $\phi : \Gamma \to Q$ denotes the quotient map, and let $n \in
\N$. Suppose that $N$ is $n$-boundedly acyclic. Then $\Gamma$ is
$n$-boundedly acyclic if and only if $Q$ is $n$-boundedly acyclic.
\end{thm}

In particular, the class of $n$-boundedly acyclic groups is closed
under extensions. This generalizes the classical results for
extensions with amenable kernel~\cite{vbc} and amenable
quotient~\cite{Johnson, coamenable}.

A natural question is then whether a $2$-out-of-$3$ property holds. Namely:

\begin{quest}
\label{q:2:3}
In an extension~$1 \to N \to \Gamma \to Q \to 1$, suppose that
$\Gamma$ and $Q$ are $n$-boundedly acyclic. Is $N$ necessarily
$n$-boundedly acyclic?
\end{quest}

A characterization of when this occurs is available~\cite[Corollary
  4.2.1]{BAc}, but it is given in terms of vanishing of bounded
cohomology with a larger class of coefficients, and so it does not
settle Question \ref{q:2:3} in either direction.  We will answer
Question~\ref{q:2:3} in the negative in Theorem~\ref{thm:normal
  subgroups}.

%%%%%%%%%%
\subsection{Normal subgroups}

Pseudo-mitotic groups are varied enough that they allow for several
strong embedding constructions. The following is the most relevant
one:

\begin{example}[Cone over a group]\label{ex:cone}
If $\Gamma$ is a group, let $\Gamma^{\Q}$ be the group of functions
from $\Q$ to $\Gamma$ that map all numbers outside some finite
interval to the neutral element. The group~$C\Gamma := \Gamma^{\Q}
\rtimes \Aut(\Q)$, where $\Aut(\Q)$ (defined in Example
\ref{ex:dissipators}) acts on $\Gamma^{\Q}$ by shifting the
coordinates, is pseudo-mitotic, in fact dissipated~\cite[Section
  3.5]{Berrick}. The group~$C \Gamma$ is called the \emph{cone over
  $\Gamma$}, and was introduced by Kan and Thurston~\cite{KanThurston}
as a key step in the proof of their celebrated theorem.
\end{example}

\begin{prop}[{\cite[Section 3.5]{Berrick}}]\label{prop:2step}
Every group embeds $2$-step subnormally in a pseudo-mitotic
group. More precisely, for every group $\Gamma$ there exists a
group~$\Gamma^0$ such that $C\Gamma \cong (\Gamma \times \Gamma^0) \rtimes
\Aut(\Q)$ is pseudo-mitotic.
\end{prop}

\begin{proof}
By Example~\ref{ex:cone}, it suffices to show that $\Gamma^{\Q} \cong
\Gamma \times \Gamma^0$, for some group~$\Gamma^0$.  The
group~$\Gamma$ embeds normally in~$\Gamma^{\Q}$ as the subgroup of
functions $\Q \to \Gamma$ that map every non-zero rational to the
identity.  We then set $\Gamma^0$ to be the subgroup of functions that
map~$0 \in \Q$ to~$1 \in \Gamma$.
\end{proof}

In particular, every group embeds $2$-step subnormally in a boundedly
acyclic group.  Embeddings into boundedly acyclic groups have been
considered before~\cite{Loeh, our}, but Proposition~\ref{prop:2step}
goes one step further, and provides a strong negative answer to
Question~\ref{q:2:3}:

\begin{thm}\label{thm:normal subgroups}
There exists a boundedly acyclic group~$\Gamma$ with a normal
subgroup~$H$ such that $\Gamma/H$ is boundedly acyclic, but
$\HH^n_b(H; \R)$ is continuum-dimensional for every~$n \geq 2$.
\end{thm}

Groups such as $H$ above are said to have \emph{large bounded
  cohomology}: Countable~\cite{Loeh} and even finitely
generated~\cite{our} examples are known to exist.

\begin{proof}
Let $H$ be a group with large bounded cohomology.  Then, for every
group~$\Lambda$, the direct product $H \times \Lambda$ also has
large bounded cohomology (as it retracts onto a group with large
bounded cohomology).  By Proposition \ref{prop:2step} and Example
\ref{ex:dissipators}, this implies that the pseudo-mititotic group
$\Gamma := C H$ provides the desired example.
\end{proof}

Even without the additional hypothesis on the quotient, it seems that
Theorem~\ref{thm:normal subgroups} gives the first example of a
non-boundedly acyclic normal subgroup of a boundedly acyclic group.
Indeed, subgroups of amenable groups are amenable, and several of the
non-amenable examples of boundedly acyclic groups available in the
literature are simple, so they cannot provide counterexamples.  For
instance, $\Homeo_c(\R^n)$ is simple, as are many other groups of
boundedly supported homeomorphisms~\cite{dissipatedc1}.

%%%%%%%%%%%%%%%%%%
\subsection{Quotients}

An intriguing open problem is whether boundedly acyclic groups are
closed under passage to quotients~\cite[Section 3.2]{our}.  One of the
main difficulties about this problem is that mitotic groups behave
extremely well with respect to quotients:

\begin{lemma}[{\cite[Page 16]{BDH}}]
\label{lem:quot:m}
Mitotic groups are closed under passage to quotients.
\end{lemma}

On the other hand, the same behaviour does not hold for pseudo-mitotic
groups:

\begin{lemma}[{\cite[Theorem~3.3]{pseudonotmitotic}}]
\label{lem:quot:pm}
Pseudo-mitotic groups are not closed under passage to quotients.
\end{lemma}

This suggests that pseudo-mitotic groups might be useful for
constructing counterexamples for the problem above. However, the
example considered by Sankaran and
Varadarajan~\cite[Theorem~3.3]{pseudonotmitotic} still produces a
boundedly acyclic quotient. Indeed, in this situation the kernel of
the epimorphism is the group of finitely supported permutations of
$\N$, which is locally finite and therefore amenable.  Hence, by
Theorem \ref{thm:ext} (or simply by Gromov's Mapping Theorem
\cite{vbc}), the quotient is a boundedly acyclic non-pseudo-mitotic
group.

A more interesting situation arises from the context of algebraic
$K$-theory.  Indeed, following Berrick~\cite{Berrick_book} given a
(unital, associative) ring~$R$ one can embed it into its cone~$CR$ as
a two-sided ideal. This leads to a short exact sequence
$$
1 \to \GL(R) \to \GL(CR) \to Q \to 1,
$$
where $Q$ is the direct general linear group over the suspension of~$R$,
usually denoted by~$\GL(SR)$~\cite[page~85]{Berrick_book}.
As discussed in Example~\ref{ex:flabby}.1, the group~$\GL(CR)$ is a flabby
group, whence pseudo-mitotic (in fact this is also a dissipated group as proved by 
Berrick~\cite[pages~84--85]{Berrick_book}). 
On the other hand, since one can compute the $K$-groups of the original ring~$R$ in terms
of the plus construction over~$\GL(SR)$, in general 
$\GL(SR)$ is far from being acyclic~\cite{Berrick_book}.
Hence, a natural question is the following:

\begin{quest}
Let $R$ be a ring. Is the group~$\GL(SR)$ boundedly acyclic?
\end{quest}

A negative answer to this question would show that 
boundedly acyclic groups are not closed under passage to quotients.
On the other hand, here we prove the following:

\begin{prop}
Let $R$ be a ring. Then, the group~$\GL(SR)$ is $3$-boundedly acyclic.
\end{prop}

\begin{proof}
Recall that a group extension provides an exact sequence in bounded cohomology
in low degrees~\cite[Corollary~12.4.1 and Example~12.4.3]{monod}, which in this case gives
\begin{align*}
0 \to \HH^2_b(\GL(SR); \R) &\to \HH^2_b(\GL(CR); \R) \to \HH^2_b(\GL(R); \R)^{\GL(SR)} \\
&\to \HH^3_b(\GL(SR); \R) \to \HH^3_b(\GL(CR); \R).
\end{align*}
Using the fact that $\GL(CR)$ is pseudo-mitotic, whence boundedly
acyclic (Theorem~\ref{thm:binatebac}), we then have
$$
\HH^2_b(\GL(SR); \R) \cong 0
$$
and
$$
\HH^2_b(\GL(R); \R)^{\GL(SR)} \cong \HH^3_b(\GL(SR); \R).
$$
We show that $\HH^2_b(\GL(R); \R) \cong 0$, whence the thesis.

It suffices to show that $\GL(R)$ has \emph{commuting
  conjugates}~\cite{cc}; that is, for every finitely generated
subgroup~$H \leq \GL(R)$ there exists~$g \in \GL(R)$ such that $H$ and
$g^{-1} H g$ commute.  Now let $H \leq \GL(R)$ be finitely
generated. Then there exists some~$n \geq 1$ such that $H \leq
\GL_n(R)$.  Let $g \in \GL_{2n}(R) \leq \GL(R)$ be a permutation
matrix that swaps the basis vectors~$e_1, \ldots, e_n$ with~$e_{n+1},
\ldots, e_{2n}$.  Then $g^{-1} H g$ acts trivially on the span
of~$e_1, \ldots, e_n$, and $H$ acts trivially on the span of~$e_{n+1},
\ldots, e_{2n}$; therefore, these subgroups commute.  We conclude that
$\GL(R)$ has commuting conjugates, and so $\HH^2_b(\GL(R); \R) \cong
0$~\cite{cc}. This finishes the proof.
\end{proof}

%%%%%%%%%%%%%%%%%%
\subsection{Directed unions}

The operations we have looked at so far are known to preserve
amenability.  This is not surprising since amenable groups are the
most illustrious examples of boundedly acyclic groups.  One further
operation that preserves amenability is that of directed unions.  Here
we study the behaviour of bounded acyclicity under directed unions,
and show that it is preserved under an additional technical
requirement.

To proceed with the proof, it is convenient to consider the following
dual version of UBC:

\begin{defi}\label{defi:vanmod}
Let $n \in \N$, and let $\Gamma$ be a group such that $\HH^n_b(\Gamma;
\R) \cong 0$.  We define the $n$-th \emph{vanishing modulus} of
$\Gamma$ as the minimal $K \in \R_{\geq 0} \cup \{\infty\}$ such that
the following holds:

For each $c \in \ker(\delta^n_b)$ there exists $b \in
\CC^{n-1}_b(\Gamma; \R)$ such that
$$\delta^{n-1}_b(b) = c \qand |b|_\infty \leq K \cdot |c|_\infty.$$
\end{defi}

\begin{example}
Every amenable group~$\Gamma$ has an $n$-th vanishing modulus of $1$,
for all $n \geq 1$.  Indeed, the proof of bounded acyclicity of
amenable groups~\cite[Theorem 3.6]{Frigerio} exhibits a contracting
chain homotopy $\tau$ for the cochain complex~$\CC^*_b(\Gamma; \R)$,
which has norm~$1$ in every degree.  Hence, given $c \in
\ker(\delta^n_b)$, we can just set $b := \tau^n(c)$ and obtain:
$$
\delta^{n-1}_b(b) = \delta^{n-1}_b \tau^n(c) + \tau^{n+1} \delta^{n}_b(c) = c.
$$
\end{example}

In our definition the vanishing modulus takes values in $\R_{\geq 0}
\cup \{\infty\}$. It turns out that only finite values are possible.

\begin{lemma}\label{lem:vanmod:finite}
  Let $n \in \N$ and let $\Gamma$
  be such that $\HH^n_b(\Gamma; \R) \cong 0$.
  Then the $n$-th vanishing modulus of $\Gamma$ is finite.
\end{lemma}

\begin{proof}
  This is implicit in the work of Matsumoto and Morita~\cite{MM}: 
  Because of~$\HH_b^n(\Gamma;\R) \cong 0$, we have~$\im \delta_b^{n-1}
  = \ker \delta_b^n$. Hence, the bounded linear map~$\delta_b^{n-1}$
  has closed range; by the Open Mapping Theorem, $\delta_b^{n-1}$
  induces a Banach space isomorphism
  \[ \overline\delta_b^{n-1}
  \colon \CC_b^{n-1}(\Gamma;\R) / \ker \delta_b^{n-1}
  \to \ker \delta_b^n.
  \]
  Let $\varphi^n$ be the inverse of~$\overline \delta_b^{n-1}$. 
  If $c \in \ker\delta_b^n$, then the definition of
  the quotient norm on~$\CC_b^{n-1}(\Gamma;\R) / \ker \delta_b^{n-1}$
  shows that there exists a~$b \in \CC_b^{n-1}(\Gamma;\R)$
  with
  \[ \delta_b^{n-1}(b) = c
  \qand
  |b|_\infty \leq 2 \cdot \|\varphi^n\| \cdot |c|_\infty.
  \]
  Thus, the constant~$2 \cdot \|\varphi^n\|$
  is a finite upper bound for the $n$-th vanishing modulus.
\end{proof}

\begin{prop}
  \label{prop:dirun}
  Let $\Gamma$ be a group that is the directed union of a directed
  family~$(\Gamma_i)_{i \in I}$ of subgroups.  Moreover, let $n \in
  \N$ and suppose that $\HH^n_b(\Gamma_i; \R) \cong 0$ for all $i$,
  and that there is a uniform, finite upper bound for the $n$-th
  vanishing moduli of all the $\Gamma_i$'s.  Then $\HH^n_b(\Gamma; \R)
  \cong 0$.
\end{prop}
\begin{proof}
  Let $K < +\infty$ be a common upper bound for the $n$-th vanishing
  moduli of the $\Gamma_i$'s.  We show that the $n$-th vanishing
  modulus of $\Gamma$ is at most $K$. Let $c \in
  \CC_b^n(\Gamma;\R)$ be a bounded cocycle.  For each~$i \in I$, we
  set
  \[ B_i := \bigl\{ b \in \CC_b^{n-1}(\Gamma;\R)
  \bigm| \delta_b^{n-1}(b|_{\Gamma_i}) = c|_{\Gamma_i}
  \text{ and } |b|_\infty \leq K \cdot |c|_\infty
  \bigr\}.
  \]
  It suffices to show that $\bigcap_{i \in I} B_i \neq \emptyset$.  To
  this end, we use the Banach--Alaoglu Theorem: By construction,
  each~$B_i$ is a bounded weak$*$-closed subset
  of~$\CC_b^{n-1}(\Gamma;\R)$ and $B_j \subset B_i$ for all~$j \in I$
  with~$i \leq j$.  Moreover, $B_i \neq \emptyset$: By hypothesis,
  there exists~$b_i \in \CC_b^{n-1}(\Gamma_i;\R)$
  with~$\delta_b^{n-1}(b_i) = c|_{\Gamma_i}$ and $|b_i|_\infty \leq K
  \cdot |c|_\infty$.  We now extend~$b_i$ by~$0$; this extension lies
  in~$B_i$.

  Because the system is directed, the family~$(B_i)_{i\in I}$ satisfies
  the finite intersection property; by the Banach--Alaoglu Theorem,
  therefore the whole intersection~$\bigcap_{i \in I} B_i$ is non-empty.
\end{proof}

The following special case will be used when studying
the Thompson group~$F$ (Lemma \ref{lem:F:F1:bac}):

\begin{cor}
\label{cor:dirun}
Let $\Gamma$ be a group that is the directed union of a directed
family $(\Gamma_i)_{i \in I}$ of subgroups.  Suppose that the
$\Gamma_i$'s are pairwise isomorphic and $n$-boundedly
acyclic.  Then $\Gamma$ is $n$-boundedly acyclic.
\end{cor}

\begin{proof}
This follows directly from Lemma~\ref{lem:vanmod:finite} and
Proposition~\ref{prop:dirun}.
\end{proof}

In degree $2$, we may get rid of the uniformity condition in
Proposition~\ref{prop:dirun}, thanks to the following surprising fact:

\begin{prop}
\label{prop:degree2uniform}
Let $\Gamma$ be a $2$-boundedly acyclic group. Then the second
vanishing modulus of~$\Gamma$ is~$1$.
\end{prop}

\begin{proof}
  This is essentially a dual version of a result by Matsumoto and
  Morita~\cite[Corollary~2.7]{MM}.
  First, the map~$\delta_b^1$ is injective, since the only bounded
  homomorphism~$\Gamma \to \R$ is the trivial one.  We consider
  the map 
  \begin{align*}
    \psi \colon \ker\delta_b^2 = \im \delta_b^1
    & \to \CC^1_b(\Gamma;\R)
    \\
    c & \mapsto
    \biggl( (g_0,g_1) \mapsto
    \sum_{k=0}^\infty 2^{-(k+1)} \cdot c (1, g_{01}^{2^k}, g_{01}^{2^{k+1}})
    \biggr),
  \end{align*}
  where we use the abbreviation~$g_{01} := g_0^{-1} \cdot g_1$,
  and claim that $\psi$ is the inverse of~$\delta_b^1$.
  By definition, $\|\psi\| \leq 1$; moreover, $\|\psi\| =1$
  because $\delta_b^1$ sends constant functions to constant functions.

  We are left to prove the claim. Let $c \in \im \delta_b^1$, say~$c =
  \delta_b^1(b)$. We need to show that $\psi(c) = b$: Using $\Gamma$-invariance,
  we obtain for all~$g_0,g_1 \in \Gamma$:
  \begin{align*}
    \bigl(\psi(c)\bigr)(g_0,g_1)
    %& =
    %\sum_{k=0}^\infty 2^{-(k+1)} \cdot \delta_b^1b (1, g_{01}^{2^k}, g_{01}^{2^{k+1}})
    %\\
    & =
    \sum_{k=0}^\infty 2^{-(k+1)} \cdot
    \bigl( b(g_{01}^{2^k}, g_{01}^{2^{k+1}}) - b (1, g_{01}^{2^{k+1}}) + b(1, g_{01}^{2^k})\bigr)
    \\
    & =
    \sum_{k=0}^\infty 2^{-(k+1)} \cdot
    \bigl( b(1, g_{01}^{2^{k}}) - b (1, g_{01}^{2^{k+1}}) + b(1, g_{01}^{2^k})\bigr)
    \\
    & =
    \sum_{k=0}^\infty 2^{-k} \cdot b(1,g_{01}^{2^k})
    - \sum_{k=0}^\infty 2^{-(k+1)} \cdot b(1,g_{01}^{2^{k+1}})
    \\
    & = b(1, g_{01}) = b(g_0,g_1).
  \end{align*}
  Note that all series involved are absolutely convergent because $b$ is bounded,
  which is what allows us to change the order of summation. 
\end{proof}

\begin{cor}
\label{cor:degree2dirun}
A directed union of $2$-boundedly acyclic groups is $2$-bound\-edly
acyclic.  \hfill\qedsymbol
\end{cor}

Proposition \ref{prop:degree2uniform} is essentially equivalent to the
fact that the canonical semi-norm in degree~$2$ is always a
norm~\cite[Corollary 2.7]{MM}.  This fails already in
degree~$3$~\cite{norm0free, norm0ah}, but such examples also have
large bounded cohomology and so are difficult to control.  Therefore
we ask:

\begin{quest}
Does the analogue of Proposition \ref{prop:degree2uniform} hold in
higher degrees?
\end{quest}

One can use Lemma~\ref{lem:vanmod:finite} to show that a direct sum of
$n$-boundedly acyclic groups with unbounded vanishing modulus cannot
be $n$-boundedly acyclic.  Therefore a negative answer to this
question would imply that, in higher degrees, the uniformity
assumption in Proposition \ref{prop:dirun} is necessary.

%%%%%%%%%%%%%%%%%%%%%%%%%%%%%%%%%%%%%%%%%%%%%%%%%%%%%%%%%%%
\section{Universal bounded acyclicity}\label{s:ubac}

In this section, we show that the bounded acyclicity of pseudo-mitotic
groups is not a phenomenon confined to real coefficients. Since
several different coefficients are involved in this section, we will
be explicit and talk about $\Z$-acyclic groups (Definition~\ref{defi:acyclic})
and $\R$-boundedly acyclic groups (Definition~\ref{defi:bac}).

\begin{defi}\label{defi:ubac}
Let $\K$ be a complete valued field, and let $\Gamma$ be a group. We
say that $\Gamma$ is \emph{$\K$-boundedly acyclic} if $\HH^n_b(\Gamma;
\K) \cong 0$ for all~$n \geq 1$. If this holds for all complete valued
fields~$\K$, we say
that $\Gamma$ is \emph{universally boundedly acyclic}.
\end{defi}

We can characterize universal bounded acyclicity in very
simple terms:

\begin{thm}
\label{thm:ubac}
Let $\Gamma$ be a group. Then $\Gamma$ is universally boundedly
acyclic if and only if it is $\R$-boundedly acyclic and $\Z$-acyclic.
\end{thm}

\begin{rem}
In fact, Theorem \ref{thm:ubac} even holds degree-wise. More
precisely, for a group $\Gamma$ and an integer $n \geq 1$ the
following are equivalent:
\begin{enumerate}
\item $\HH^i_b(\Gamma; \R) \cong 0$ and $\HH_i(\Gamma; \Z) \cong 0$
  for all $i \in \{1, \ldots, n\}$;
\item $\HH^i_b(\Gamma; \K) \cong 0$ for every complete valued field
  $\K$ and all $i \in \{1, \ldots, n\}$.
\end{enumerate}
This will be apparent from the proof, but we prefer to state the
theorem in global terms to simplify the notation.
\end{rem}

Before giving the proof, we note the following consequence:

\begin{cor}
\label{cor:ubac}
Pseudo-mitotic groups are universally boundedly acyclic.
\end{cor}
\begin{proof}
  Pseudo-mitotic groups are both acyclic
  (Theorem~\ref{thm:pseudo:mitotic:are:acyclic}) and boundedly acyclic
  (Theorem~\ref{thm:pseudo:mit:bac}).  Therefore, we can apply
  Theorem~\ref{thm:ubac}.
\end{proof}

The proof of Theorem \ref{thm:ubac} will be carried out in two steps:
the Archimedean and the non-Archimedean case.

\begin{lemma}
\label{lem:ubac:arch}
Let $\Gamma$ be a group. Then $\Gamma$ is $\R$-boundedly acyclic if
and only if $\Gamma$ is $\C$-boundedly acyclic.
\end{lemma}

\begin{proof}
  Because $\C \cong \R^2$ as normed $\R$-vector spaces,
  the cochain complex~$\CC_b^*(\Gamma;\C)$ splits
  as the direct sum~$\CC_b^*(\Gamma;\R)^{\oplus2}$.
  Therefore, we obtain the isomorphism 
  $\HH_b^*(\Gamma;\C) \cong \HH_b^*(\Gamma;\R)^{\oplus 2}
  $
  (over~$\R$). 
  The claim easily follows.
\end{proof}

\begin{lemma}
\label{lem:ubac:nonarch}
Let $\Gamma$ be a group. Then $\Gamma$ is $\K$-boundedly acyclic for
every non-Archimedean field $\K$ if and only if it is $\Z$-acyclic.
\end{lemma}

\begin{proof}
Suppose that $\Gamma$ is $\K$-boundedy acyclic for every
non-Archimedean field $\K$. Endowing an arbitrary field $\K$ with the
trivial norm, we deduce that $\HH^n(\Gamma; \K) \cong 0$ for every
field $\K$. It then follows from the Universal Coefficient
Theorem~\cite[Chapter~I]{Brown} that $\Gamma$ is $\K$-acyclic for
every field $\K$, that is, $\HH_n(\Gamma; \K) \cong 0$ for all $n \geq
1$.  In particular, $\Gamma$ is $\Q$-acyclic and $\Fp$-acyclic for
every prime $p$; so $\Gamma$ is
$\Z$-acyclic~\cite[Corollary~3A.7]{Hatcher}.

Conversely, let us suppose that $\Gamma$ is $\Z$-acyclic, and let $\K$
be a non-Archimedean field $\K$.  By Lemma \ref{lem:nonarch}, the
comparison map $\HH^n_b(\Gamma; \K) \to \HH^n(\Gamma; \K)$ is
injective.  So it suffices to show that $\HH^n(\Gamma; \K) \cong 0$.
This follows immediately from the Universal Coefficient Theorem;
hence, $\Gamma$ is $\K$-boundedly acyclic.
\end{proof}

\begin{proof}[Proof of Theorem \ref{thm:ubac}]
By Ostrowski's Theorem~\cite[Chapter Three]{Cassels}, every complete
valued field is either non-Archimedean or isomorphic to $\R$ or~$\C$.
So Theorem~\ref{thm:ubac} follows from Lemmas~\ref{lem:ubac:arch}
and~\ref{lem:ubac:nonarch}.
\end{proof}

\begin{rem}
Corollary~\ref{cor:ubac} provides many examples of groups that are
universally boundedly acyclic.  One could ask whether something
similar could be said for the stronger notion of \emph{universal
  amenability}, defined analogously using the general notion of
$\K$-amenability for valued fields defined by Shikhof~\cite{Schikhof}.
However, it turns out that if $\Gamma$ is $\Fp$-amenable in the sense
of Shikhof for every prime $p$, then $\Gamma$ is
trivial~\cite[Example~5.5, Theorem~6.2]{nonarch}.  The same holds for
the weaker notion of normed $\K$-amenability~\cite{nonarch}, which
also implies bounded $\K$-acyclicity~\cite[Theorem 1.3]{nonarch}.
\end{rem}

\begin{cor}
  Let $\Gamma$ be a universally boundedly acyclic group.
  Then, for all~$n \geq 1$, we have
  $\HH_b^n(\Gamma;\Z) \cong 0$, with the standard
  absolute value on~$\Z$.
\end{cor}
\begin{proof}
  The short exact sequence~$0 \to \Z \to \R \to \R/\Z \to 0$
  induces a long exact sequence~\cite[proof of Proposition~2.13]{Frigerio}
  \[ \dots
     \to \HH_b^{n-1} (\Gamma;\R)
     \to \HH^{n-1}(\Gamma;\R/\Z)
     \to \HH_b^n(\Gamma;\Z)
     \to \HH_b^n(\Gamma;\R)
     \to
     \dotsm .
  \]
  By Theorem~\ref{thm:ubac}, the group~$\Gamma$ is $\R$-boundedly
  acyclic and $\Z$-acyclic. The Universal Coefficient Theorem
  and $\Z$-acyclicity give $\HH^k(\Gamma;\R/\Z) \cong 0$
  for all~$k >0$. Therefore, the long exact sequence
  and surjectivitiy of the induced map~$\HH_b^0(\Gamma;\R)
  \to \HH^0(\Gamma;\R/\Z)$ show that $\HH_b^n(\Gamma;\Z) \cong 0$
  for all~$n\geq 1$.
\end{proof}

%%%%%%%%%%%%%%%%%%%%%%%%%%%%%%%%%%%%%%%%%%%%%%%%%%%%%%%%%%%
\section{Thompson groups and their siblings}\label{s:thompson}

The groups $F$, $T$, and $V$ were introduced by Richard Thompson
in~1965; they are some of the most important groups in geometric and
dynamical group theory. These groups can be realized as groups of
homeomorphisms of the interval, the circle, and the Cantor set
respectively; these realizations exhibit inclusions $F \leq T \leq
V$. We refer the reader to the literature~\cite{thompson} for a
detailed discussion.

The groups~$F$, $T$, and~$V$ are finitely presented, even of
type~$F_\infty$. Moreover, $T$ and $V$ are simple (in fact, they were the
first examples of infinite finitely presented simple groups). On the
other hand, $F$ has abelianization~$\Z^2$, but its derived
subgroup~$F'$ is simple, and infinitely generated. 

The rational cohomology~\cite{cohoT, cohoV2} and, with the exception
of $T$, the integral cohomology~\cite{cohoF, cohoT, cohoV} of these
groups has been computed.  However, little is known about their real
bounded cohomology. We formulate one question for each group:

\begin{quest}
\label{qF}
Is the Thompson group~$F$ boundedly acyclic?
\end{quest}

Question~\ref{qF} is usually attributed to Grigorchuk~\cite[p.~131,
  Problem~3.19]{grigorchuk}.

\begin{quest}
\label{qT}
Does the following hold?

The bounded cohomology of the Thompson group~$T$ is given by
$$\HH^n_b(T; \R) \cong
\begin{cases}
0 & \text{if } n \text{ is odd}; \\
\R & \text{if } n \text{ is even};
\end{cases}$$
where the non-trivial classes are spanned by cup powers of the
bounded real Euler class.
\end{quest}

\begin{quest}
\label{qV}
Is the Thompson group $V$ boundedly acyclic?
\end{quest}

The rest of this section is devoted to discussing these three
questions, how they relate to each other, and provide some evidence
towards positive answers. 

One may also formulate corresponding questions for every degree,
namely whether the previous descriptions hold up to degree~$n$. We
will see that all three questions have a positive answer up to
degree~$2$, while to our knowledge nothing is known from degree~$3$
onwards.

%%%%%%%%%%
\subsection{On the bounded cohomology of~$F$}

We recall the definition of~$F$.

\begin{defi}
  The Thompson group~$F$ is the group of orientation-pre\-serv\-ing
  piecewise linear homeomorphisms~$f$ of the interval~$[0, 1]$ with
  the following properties:
\begin{enumerate}
\item $f$ has finitely many breakpoints, all of which lie in $\Z[1/2]$;
\item Away from the breakpoints, the slope of $f$ is a power of $2$.
\end{enumerate}
\end{defi}

The map $F \to \Z, f \mapsto \log_2(f_0)$, where $f_0$ is the slope
of~$f$ at~$0$, is a surjective homomorphism, called the \emph{germ at~$0$}.
Similarly, there is a germ at~$1$, leading to a surjective
homomorphism~$F \to \Z^2$. This is the abelianization of~$F$, so the
derived subgroup~$F'$ coincides with the subgroup of homeomorphisms
that are compactly supported in~$(0, 1)$.

The most important open question about~$F$ is whether $F$ is amenable
or not. Since amenable groups are boundedly acyclic, a negative answer
to Question~\ref{qF} would disprove its amenability. The
general philosophy is that $F$ is very close to being amenable, and so
it is likely to satisfy most properties that are somewhat weaker than
amenability.

For example: The group~$F$ is $2$-boundedly acyclic. This can be
deduced from the explicit description of its rational
cohomology~\cite{cohoT}, by using arguments analogous to those of
Heuer and L\"oh for the computation of the second bounded cohomology
of~$T$~\cite{spectrum} (although a direct approach is
possible~\cite{cc}). To our knowledge, nothing is known about the
bounded cohomology of~$F$ with trivial real coefficients in higher
degrees, although vanishing is known in every degree with mixing
coefficients \cite{monod_sarithmetic}.

The connection between pseudo-mitotic groups and~$F$ is more
transparent when passing to the derived subgroup. The
following equivalent formulation will be relevant:

\begin{lemma}\label{lem:F:F1:bac}
 Let $n \in \N$. Then the 
 Thompson group~$F$ is $n$-boundedly acyclic if and only if $F'$ is
 $n$-boundedly acyclic.
\end{lemma}

\begin{proof}
If $F'$ is $n$-boundedly acyclic, then $F$ is also boundedly acyclic
by Theorem~\ref{thm:ext}, or more simply by coamenability~\cite{coamenable}.

Conversely, let us suppose that $F$ is $n$-boundedly acyclic. Let
$(a_i)_{i \geq 1}$ and $(b_i)_{i \geq 1}$ be sequences of dyadic
rationals in~$(0, 1)$ that converge to $0$ and~$1$ respectively. Then
$F'$ may be expressed as the directed union of the subgroups~$F_i$
consisting of elements supported in~$[a_i, b_i]$.  Since each
group~$F_i$ is isomorphic to~$F$, the group~$F'$ is a directed union
of pairwise isomorphic $n$-boundedly acyclic groups.  It follows from
Corollary~\ref{cor:dirun} that $F'$ is $n$-boundedly acyclic.
\end{proof}

The derived subgroup~$F'$ is a group of boundedly supported
homeomorphisms of the interval. In analogy with
Example~\ref{ex:dissipators}, one may ask whether $F'$ is
pseudo-mitotic. This is not the case, because $F'$ is not
acyclic~\cite{cohoT}. Intuitively, $F'$ cannot be dissipated, since a
dissipator could not possibly have finitely many breakpoints. However,
a \emph{countably singular} analogue of~$F'$ is dissipated:

\begin{defi}
Let $\Omega F$ be the group of orientation-preserving
homeomorphisms~$f$ of the interval~$[0, 1]$ with the following
properties:
\begin{enumerate}
\item There exists a closed and countable set~$K \subset (0, 1) \cap \Z[1/2]$ such
  that $f$ is linear on each component of~$[0, 1] \setminus K$;
\item Away from~$K$, the slope of~$f$ is a power of~$2$.
\end{enumerate}
\end{defi}

Since the set of breakpoints of each element is contained in~$(0,
1)$, the germs at $0$ and $1$ are still defined, and $\Omega F'$ is
the subgroup of homeomorphisms that are compactly supported in~$(0,
1)$.

\begin{prop}
\label{prop:OF}
The groups $\Omega F$ and $\Omega F'$ are boundedly acyclic.
\end{prop}

\begin{proof}
Once again, the bounded acyclicity of $\Omega F$ follows from that of
$\Omega F'$ by Theorem \ref{thm:ext}.

We show that $\Omega F'$ is dissipated. Let $(a_i)_{i \geq 1}$ and
$(b_i)_{i \geq 1}$ be sequences of dyadic rationals in~$(0, 1)$
converging to $0$ and $1$, respectively. For every~$i \geq 1$, let
$H_i \leq \Omega F'$ be the subgroup consisting of homeomorphisms
supported in $(a_i, b_i)$.  We show that there exists a
dissipator~$\rho_i \in \Omega F'$ for~$H_i$, that is
\begin{enumerate}
\item\label{i:dis1} For every $k \geq 1$, we have $\rho_i^k((a_i,
  b_i)) \cap (a_i, b_i) = \emptyset$;

\item\label{i:dis2} For every $g \in H_i$, the element
  $$\varphi_i(g)
  :=
  \begin{cases} \rho_i^k g \rho_i^{-k} & \text{on } \rho^k(a_i,
    b_i), \mbox{ for every } k \geq 1;
    \\ \id & \text{elsewhere}
  \end{cases}
  $$
is in $\Omega F'$.
\end{enumerate}
To this end, let us set $x_0 := a_i$ and $x_1 := b_i$. We then pick a
dyadic rational $x_{-1}$ in $(0, x_0)$ such that
$$
\frac{x_1 - x_{-1}}{x_0 - x_{-1}}
$$
is a power of~$2$. Moreover, given a dyadic rational~$x \in \,
(x_1, 1)$, we can extend $x_{-1}, x_0, x_1$ to a sequence~$(x_j)_{j
  \geq -1}$ of dyadic rationals converging to $x$ and such that for
every $j \geq 1$ the ratio
$$
\frac{x_{j+1}- x_j}{x_j - x_{j-1}}
$$
is a power of~$2$.

Now we define $\rho_i \colon [0, 1] \to [0, 1]$ piecewise as follows:
\begin{align*}
  \rho_i |_{[0, x_{-1}] \cup [x, 1]}
  & := \id|_{[0, x_{-1}] \cup [x, 1]},\\
  \rho_i([x_{-1}, x_0])
  & := [x_{-1}, x_1] \\
  \rho_i([x_{j-1}, x_j])
  & := [x_j, x_{j+1}] \text{ for } j \geq 1
\end{align*}
and let $\rho_i$ be the unique affine isomorphism on each of these
pieces.  Notice that $\rho_i$ is supported in $[x_{-1}, x] \subset
(0,1)$, and the set~$\{ x, x_{-1}, x_0, x_1, \ldots \}$ of breakpoints
is closed, countable, and consists only of dyadic rationals. Since all
the slopes are powers of $2$, this implies that $\rho_i \in \Omega
F'$.

We claim that $\rho_i$ is a dissipator for $H_i$. First, notice that
by construction and the definition of $x_0$ and $x_1$, we have 
$$
\rho_i^k(a_i, b_i) \cap (a_i, b_i) = (x_k, x_{k+1}) \cap (x_0, x_1) = \emptyset
$$
for every $k \geq 1$. This shows that $\rho_i$ satisfies property~(\ref{i:dis1}).

Finally, for every $g \in \, H_i$, the support of the homeomorphism
$\varphi_i(g)$ is contained in $[x_0, x]$.  Moreover, the set of
breakpoints of $\varphi_i(g)$ is still a closed, countable set
consisting only of dyadic rationals.  Hence, $\rho_i$ also satisfies
property~(\ref{i:dis2}).

This shows that $\Omega F'$ is dissipated, whence pseudo-mitotic by
Proposition~\ref{prop:dissipated}. The thesis now follows from
Theorem~\ref{thm:pseudo:mit:bac}.
\end{proof}

We believe that a careful study of the embedding~$F' \hookrightarrow
\Omega F'$ could lead to some understanding of the bounded cohomology
of~$F'$.

%%%%%%
\subsection{On the bounded cohomology of $T$}

We recall the definition of $T$:

\begin{defi}
  The Thompson group~$T$ is the group of orientation-pre\-ser\-ving
  piecewise linear homeomorphisms~$f$
  of the circle~$\R/\Z$ with the following properties:
\begin{enumerate}
\item $f$ has finitely many breakpoints, all of which lie in $\Z[1/2]/\Z$;
\item Away from the breakpoints, the slope of $f$ is a power of $2$;
\item $f$ preserves $\Z[1/2]/\Z$.
\end{enumerate}
\end{defi}

The stabilizer of~$0$ for the canonical $T$-action on the circle is
canonically isomorphic to the Thompson group~$F$.

Since $T$ acts minimally on the circle, it admits a second bounded
cohomology class, namely the \emph{real Euler
  class}~\cite{matsumoto,semiconjugacy}.  The bounded Euler class is a
refinement of the classical Euler class. All cup powers of the
classical Euler class are non-trivial in cohomology~\cite{cohoT};
thus, also the cup-powers of the bounded Euler class are non-trivial
in~$\HH_b^*(T;\R)$; this was first noticed by Burger and
Monod~\cite{rigidity}.  Therefore, Question~\ref{qT} is asking whether
these are the only bounded cohomology classes. In degree~$2$, this is
known to be true~\cite{spectrum}, but again, to our knowledge nothing
is known in higher degrees.

The main goal of this section is to show that a positive answer to
Question~\ref{qF} implies a positive answer to Question~\ref{qT}, and
this implication holds degree-wise. In order to do this, we prove the
following general criterion for computing the bounded cohomology of
groups acting highly transitively on the circle with boundedly acyclic
stabilizers:

\begin{prop}
  \label{prop:circle}
  Let $n \in \N_{\geq 2}$.  Let $\Gamma$ be a group acting
  orientation-preserv\-ing\-ly on the circle, let $S$ be an orbit of
  $\Gamma$ with~$|S| \geq n+1$. Suppose that the following holds:
\begin{enumerate}
\item For all~$k \in \{1,\dots,n+1\}$, the action of~$\Gamma$ on the set of
  circularly ordered $k$-tuples in~$S$ is transitive;
\item For all~$k \in \{1,\dots,n\}$, the stabilizer of a circularly ordered
  $k$-tuple is $n$-boundedly acyclic.
\end{enumerate}
Then $\HH_b^2(\Gamma;\R)$ is generated
by the bounded Euler class of this circle action of~$\Gamma$ and 
$$\HH^i_b(\Gamma; \R) \cong
\begin{cases}
0  & \text{if $i$ is odd};\\
\R & \text{if $i$ is even};
\end{cases}$$
for all $i \in \{1, \ldots, n\}$, generated by the cup-powers of
Euler class.
\end{prop}

Recall that a $k$-tuple~$(s_1, \dots, s_k)$ in~$S^1$ is \emph{circularly
ordered} if there exists a point $p \in S^1 \setminus \{s_1, \dots, s_k\}$
such that $(s_1, \dots, s_k) \in S^1 \setminus \{p\} \cong (0,1)$ 
is an ordered $k$-tuple in the interval.
We follow the convention that circularly ordered tuples are
\emph{non-degenerate}, i.e., they consist of pairwise distinct entries.

  For the proof, we follow the general principle of computing bounded
  cohomology through boundedly acyclic actions. Boundedly acyclic
  stabilizers lead to boundedly acyclic modules:

  \begin{lemma}\label{lem:bacaction}
    Let $\Gamma$ be a group and let $\Gamma \actson X$ be an action
    of~$\Gamma$ on a set~$X$ that has only finitely many orbits~$(X_i)_{i
      \in I}$.
    Let $n \in \N$. If each of the orbits has $n$-boundedly acyclic stabilizer,
    then we have for all~$k \in \{1,\dots,n\}$ that
    \[ \HH_b^k\bigl(\Gamma; \linf (X) \bigr) \cong 0.
    \]
  \end{lemma}
  \begin{proof}
    Let $k \in \{1,\dots,n\}$. 
    Because $I$ is finite, 
    we have~$\linf (X) \cong \bigoplus_{i \in I} \linf(X_i)$ and
    \[ \HH_b^k\bigl(\Gamma;\linf (X) \bigr)
    \cong \HH_b^k\Bigl(\Gamma;\bigoplus_{i \in I} \linf(X_i) \Bigr)
    \cong \bigoplus_{i \in I} \HH_b^k\bigl(\Gamma;\linf(X_i)\bigr).
    \]
    We show that each of the summands is trivial. Let $i \in I$
    and let $H_i \subset \Gamma$ be the stabilizer of a point in~$X_i$.
    Then, by the Eckmann--Shapiro Lemma in bounded
    cohomology~\cite[Proposition~10.13]{monod}, we obtain
    \[ \HH_b^k\bigl(\Gamma;\linf(X_i)\bigr)
    \cong \HH_b^k\bigl(\Gamma;\linf(\Gamma)^{H_i}\bigr)
    \cong \HH_b^k(H_i;\R);
    \]
    the last term is trivial, because $H_i$ is $n$-boundedly acyclic
    by hypothesis.
  \end{proof}

  The effect of boundedly acyclic stabilizers is studied more
  systematically in a forthcoming article on boundedly acyclic
  covers and relative simplicial volume~\cite{liloehmoraschini}.

\begin{proof}[Proof of Proposition~\ref{prop:circle}]
  The given $\Gamma$-action on~$S$ gives a simplicial
  $\Gamma$-res\-o\-lu\-tion~$\R \to \linf
  (S^{*+1})$~\cite[Lemma~4.21]{Frigerio}.

  \begin{claim}\label{claim:bac}
    The $\Gamma$-resolution~$\R \to \linf (S^{*+1})$ is boundedly
    acylic up to degree~$n$, i.e., for all~$k \in \{0,\dots,n\}$
    and all $i \in \{1,\dots,n\}$, we have
    \[ \HH_b^i\bigl(\Gamma;\linf (S^{k+1})\bigr) \cong 0.
    \]
  \end{claim}
  \begin{proof}[Proof of Claim~\ref{claim:bac}]
    The $\Gamma$-space~$S^{k+1}$ consists only of finitely many
    $\Gamma$-orbits:
    Indeed, every tuple can be permuted to be circularly ordered
    (possibly with repetitions), and
    only finitely many permutations and repetition patterns
    are possible. Moreover, $\Gamma$ acts transitively on circularly
    ordered tuples of every given size~$\leq k+1$.

    The stabilizer groups of the $\Gamma$-space~$S^{k+1}$
    are all $n$-boundedly acyclic by hypothesis. 
    Thus, Lemma~\ref{lem:bacaction} shows
    the claim.
  \end{proof}

  Therefore, we can apply the fact that boundedly acyclic resolutions
  compute bounded cohomology~\cite[Proposition~2.5.4]{BAc} and
  symmetrisation~\cite[Section 4.10]{Frigerio} to conclude that
  \begin{align}
    \HH_b^i(\Gamma;\R)
  \cong \HH^i \bigl( \linf(S^{*+1})^\Gamma \bigr)
  \cong \HH^i \bigl( \linfalt(S^{*+1})^\Gamma \bigr)
  \label{eq:bcT}
  \end{align}
  for all~$i \in \{1,\dots,n\}$. Here, $\linfalt(S^{*+1})$ denotes
  the subcomplex of alternating cochains, i.e., functions~$f$
  with
  \[ f(s_{\sigma(0)}, \dots, s_{\sigma(k)}) = \sgn(\sigma) \cdot f(s_0, \dots, s_k)
  \]
  for all~$(s_0, \dots, s_k) \in S^{k+1}$ and all
  permutations~$\sigma$ of $\{0, \ldots, k\}$.
  
  \begin{claim}\label{claim:alt}
    Let $k \in \{0,\dots,n\}$.
    \begin{enumerate}
    \item If $k$ is odd, then $\linfalt(S^{k+1})^\Gamma \cong 0$.
    \item If $k$ is even, then $\linfalt(S^{k+1})^\Gamma \cong \R$,
      generated by the function~$f_k$ constructed in the proof below.  
    \end{enumerate}
  \end{claim}
  \begin{proof}[Proof of Claim~\ref{claim:alt}]
    Let $f \in \linfalt(S^{k+1})$. We first show that $f$ is
    determined by the value on a single circularly ordered tuple:
    Indeed, $f$ vanishes on tuples with a repetition; all other tuples
    may be permuted to be circularly ordered. Moreover, since $\Gamma$
    acts transitively on the set of circularly ordered tuples, $f$ is
    constant on the set of all circularly ordered tuples. In
    particular, $\dim_\R \linfalt(S^{k+1})^\Gamma \leq 1$.

    As $|S| \geq n+1$, there exists a circularly ordered
    tuple~$(s_0, \dots, s_k) \in S^{k+1}$. 
    
    Let $k$ be odd. It suffices to show that $f(s_0, \dots, s_k)=0$.
    With $(s_0, \dots, s_k)$ also $(s_k, s_0, \dots, s_{k-1})$
    is circularly ordered. Because $k$ is odd, these two tuples differ
    by an odd permutation. As $f$ is both constant on all circulary
    ordered tuples and alternating, we obtain
    \[ f(s_0, \dots, s_k)
    = f(s_k, s_0, \dots, s_{k-1})
    = -f(s_0, \dots, s_k)
    \]
    and thus $f(s_0,\dots,s_k) =0$. Therefore,
    $\linfalt(S^{k+1})^\Gamma \cong 0$.

    Let $k$ be even. We define~$f_k \colon S^{k+1} \to \R$ as follows:
    On tuples with a repetition, we define~$f_k$ to vanish. If $(t_0,
    \dots, t_k) \in S^{k+1}$ has no repetition, we set
    \[ f_k(t_0, \dots, t_k) := \sgn(\sigma),
    \]
    where $\sigma$ is a permutation such that $(t_{\sigma(0)}, \dots,
    t_{\sigma(k)})$ is circularly ordered; this permutation~$\sigma$
    is only unique up to a $(k+1)$-cycle, but since $k$ is even,
    $\sgn(\sigma)$ is well-defined. Because $\Gamma$ acts
    orientation-preservingly and because there exists at least one
    circularly ordered $(k+1)$-tuple, this gives a well-defined non-trivial
    element in~$\linfalt(S^{k+1})^\Gamma$. Therefore,
    $\linfalt(S^{k+1})^\Gamma \cong \R$.
  \end{proof}

  In view of Claim~\ref{claim:alt}, the cochain
  complex~$\linfalt(S^{*+1})^\Gamma$ is (up to degree~$n$) isomorphic
  to the cochain complex
  \[ \R \to 0 \to \R \to 0 \to \dots  
  \]
  (whose coboundary operator is necessarily trivial) and  
  if $k \in \{0,\dots, n\}$ is even, then $[f_k]$ is non-trivial
  in~$\HH^k(\linfalt(S^{*+1})^\Gamma)$. 
  In particular, we obtain
  \[
  \HH_b^i(\Gamma;\R)
  \cong \HH^i\bigl(\linfalt(S^{*+1})^\Gamma \bigr)
  \cong
  \begin{cases}
    0 & \text{if $i$ is odd};\\
    \R & \text{if $i$ is even}
  \end{cases}
  \]
  for all~$i \in \{0,\dots, n-1\}$.
  
  As for degree~$n$, under our assumptions we cannot show that
  $\linfalt(S^{n+1})$ follows the same periodic pattern.  However we
  still have:
  
 \begin{claim}\label{claim:lastdifferential}
 The differential $\linfalt(S^{n+1})^{\Gamma} \to
 \linfalt(S^{n+2})^{\Gamma}$ is trivial.
 \end{claim}
 
 \begin{proof}
 This is obvious if $n$ is odd, since then $\linfalt(S^{n+1})^{\Gamma}
 \cong 0$ by Claim~\ref{claim:alt}.
 
 Suppose instead that $n$ is even, and let $f \in
 \linfalt(S^{n+1})^{\Gamma}$ be a function that takes the constant
 value $\lambda$ on circularly ordered tuples.  Then, if $(s_0,
 \ldots, s_{n+1})$ is a circularly ordered tuple, we have
 $$\delta^{n+1}_b(f)(s_0, \ldots, s_{n+1}) = \sum\limits_{i = 0}^{n+1} (-1)^i \cdot \lambda = 0,$$
 since $n$ is even.
 \end{proof}
 
 Therefore
 \[
  \HH_b^n(\Gamma;\R) \cong
  \begin{cases}
    0 & \text{if $n$ is odd};\\
    \R & \text{if $n$ is even}
  \end{cases}
  \]
  as well.   
  It remains to deal with the bounded Euler class and its powers:

  \begin{claim}\label{claim:eu}
    The bounded Euler class~$\eub \Gamma \in
    \HH_b^2(\Gamma;\R)$ is non-trivial.
  \end{claim}
  \begin{proof}[Proof of Claim~\ref{claim:eu}]
    We make the
    isomorphism in~\eqref{eq:bcT} more explicit. Let $x_0 \in S$.
    For~$k\in \N$, we consider the map
    \begin{align*}
      \varphi^k \colon \linf(S^{*+1})
      & \to \linf(\Gamma^{*+1})
      \\
      f
      & \mapsto
      \bigl((\gamma_0, \dots, \gamma_k)
      \mapsto f(\gamma_0x_0, \dots, \gamma_k x_0)\bigr).
    \end{align*}
    Then $\varphi^* \colon \linf(S^{*+1}) \to \linf(\Gamma^{*+1})$ is
    a degree-wise bounded $\Gamma$-cochain map that extends the
    identity on the resolved module~$\R$.  Because the
    resolution~$\linf(S^{*+1})$ is strong~\cite[Lemma~4.21]{Frigerio},
    $(\varphi^*)^\Gamma$ induces an isomorphism~$\HH_b^i(\Gamma;\R)
    \cong \HH^i(\linf(S^{*+1})^\Gamma)$ for all~$i \in \{0,\dots,
    n\}$~\cite[Proposition~2.5.4, Remark~2.5.5]{BAc}.

    As the inclusion~$i^* \colon \linfalt(S^{*+1}) \to \linf(S^{*+1})$
    is a $\Gamma$-cochain map that induces an
    isomorphism~$\HH^*(\linfalt(S^{*+1})^\Gamma) \cong
    \HH^*(\linf(S^{*+1})^\Gamma)$, we conclude that $\varphi_{\alt}^*
    := \varphi^* \circ i^*$ induces an isomorphism $\HH_b^i(\Gamma;\R)
    \cong \HH^i(\linfalt(S^{*+1})^\Gamma)$.

    By construction, $(\varphi_{\alt}^2)^\Gamma(f_2)$ gives the
    orientation cocycle~$\orc^\Gamma$ of the $\Gamma$-action. Because
    of $[f_2] \neq 0$, we know that the bounded Euler class 
    \[ \eub \Gamma
    = \frac12 \cdot [\orc^\Gamma]
    = \frac12 \cdot \HH^2\bigl((\varphi_{\alt}^*)^\Gamma\bigr)[f_2]
    \]
    is non-zero in~$\HH_b^2(\Gamma;\R)$. 
  \end{proof}

  Because of Claim~\ref{claim:eu} and $\HH_b^2(\Gamma;\R) \cong \R$,
  we conclude that $\HH_b^2(\Gamma;\R)$ is generated by the bounded
  Euler class.

  \begin{claim}\label{claim:eucup}
    For each~$k \in \{0,\dots, n/2\}$, the cup-power~$(\eub \Gamma{})^{\cup k}
    \in \HH_b^{2k}(\Gamma;\R)$ is non-trivial.
  \end{claim}
  \begin{proof}[Proof of Claim~\ref{claim:eucup}]
    In view of the relation between $\eub \Gamma$ and $[f_2]$
    (proof of claim~\ref{claim:eu}) and the above description
    of~$\HH_b^{2k}(\Gamma;\R)$, 
    it suffices to show that $\alt(f_2{}^{\cup k})$ is non-trivial,
    where $\args\cup\args$ denotes the standard cup-product
    on the cochain level (notice that even if $f_2$ is alternating,
    the non-trivial cup-product $f_2{}^{\cup k}$ is not so). Indeed, we have
    \[ \alt(f_2{}^{\cup k}) = \frac{2^k \cdot k!}{(2k)!} \cdot f_{2k}
    \]
    (Appendix~\ref{appx:eupowers}), which is non-trivial
    (Claim~\ref{claim:alt}).
  \end{proof}

  This completes the proof of Proposition~\ref{prop:circle}.
\end{proof}  

\begin{rem}
The second hypothesis in Proposition~\ref{prop:circle} was used to
show that the modules~$\ell^\infty(S^{k+1})$ are $n$-boundedly
acyclic, which in turn is used to apply the computation of bounded
cohomology through acyclic resolutions~\cite[Proposition 2.5.4]{BAc}.
Note however that this result does not require $n$-bounded acyclicity
of all stabilizers.  Indeed, it is enough to ask that the stabilizer
of a circularly ordered $k$-tuple is $(n - k + 1)$ boundedly acyclic,
for~$k \in \{ 1, \ldots, n \}$. To keep notation simple, we chose to
state Proposition~\ref{prop:circle} with the stronger hypothesis.
\end{rem}
  
We apply this to $\Gamma = T$, to show that
if $F$ is $n$-boundedly acyclic, then Question~\ref{qT}
has a positive answer up to degree~$n$:

\begin{cor}\label{cor:HbT}
  If $F$ is boundedly acyclic, then
  $\HH^*_b(T;\R)$ (with the cup-product structure) is isomorphic
  to the polynomial ring~$\R[x]$ with~$|x| =2$,
  and the bounded
  Euler class of~$T$ is a polynomial generator of~$\HH^*_b(T;\R)$.
\end{cor}

\begin{proof}
  For each~$k \geq 1$, the group~$T$ acts transitively on the set of
  circularly ordered $k$-tuples in~$\Z[1/2]/\Z$; the stabilizers of
  this action are isomorphic to direct powers of~$F$~\cite{thompson}.
  In particular, the stabilizers are boundedly acyclic
  by Theorem~\ref{thm:ext}. 
  Therefore, Proposition~\ref{prop:circle} is applicable and we obtain
  that $\HH_b^*(T;\R)$ is isomorphic as a graded $\R$-algebra 
  to~$\R[x]$ with $x$ corresponding to the bounded Euler class~$\eub T$.

  Alternatively, in this case, the non-triviality of the powers of the
  bounded Euler class is already known through the computations of
  Ghys--Sergiescu and Burger--Monod~\cite{rigidity}.
\end{proof}

Corollary~\ref{cor:HbT} also holds in a range up to~$n$ (with the same
proof).

\begin{cor}\label{cor:HbTr}
  If $F$ is boundedly acyclic and $r \in \N_{\geq 1}$, then
  \begin{align*}
    \R[x_1,\dots,x_r] & \to \HH_b^*(T^{\times r};\R) \\
    x_j & \mapsto \HH_b^2(\pi_j;\R) (\eub T)
  \end{align*}
  defines an isomorphism of graded $\R$-algebras; here,
  $\HH_b^*(T^{\times r};\R)$ carries the cup-product structure, $|x_j|
  =2 $, and $\pi_j \colon T^{\times r} \to T$ denotes the projection
  onto the $j$-th factor for each~$j \in \{1,\dots, r\}$.  Moreover,
  the canonical semi-norm on~$\HH_b^*(T^{\times r};\R)$ then is a
  norm.
\end{cor}
\begin{proof}
  We combine Corollary~\ref{cor:HbT} with suitable K\"unneth
  arguments. As usual in bounded cohomology, some care is
  necessary to execute this.

  We first show that the polynomial ring embeds into~$\HH_b^*(T;\R)$: 
  The $\R$-algebra homomorphism~$\R[x] \to \HH^*(T;\R)$
  given by
  \[ x \mapsto \eu^T
  \]
  is injective~\cite{cohoT}. 
  Therefore, the K\"unneth Theorem shows that
  \[ x_j \mapsto \HH^2(\pi_j;\R)(\eu^T)
  \]
  yields an injective $\R$-algebra homomorphism~$\Phi^r \colon \R[x_1,\dots,x_r]
  \to \HH^*(T^{\times r};\R)$. In combination with the
  Universal Coefficient Theorem, we obtain: For every polynomial~$p \in
  \R[x_1,\dots,x_r] \setminus \{0\}$, there exists a class~$\alpha_p
  \in \HH_*(T^{\times r};\R)$ with
  \[ \bigl\langle \Phi^r(p), \alpha_p \rangle = 1.
  \]
  These non-trivial evaluations show that also the bounded version
  \begin{align*}
    \Phi_b^r \colon \R[x_1,\dots,x_r]
    & \to \HH_b^*(T^{\times r};\R)
    \\
    x_j & \mapsto \HH_b^2(\pi_j;\R)(\eub T)
  \end{align*}
  is injective; even more, for each~$p \in \R[x_1,\dots,x_r]\setminus \{0\}$,
  we have
  \[ \bigl\langle \Phi_b^r(p) , \alpha_p \bigr\rangle
  = \bigl\langle \Phi^r(p), \alpha_p \bigr\rangle
  = 1
  \]
  and thus $\|\Phi_b^r(p)\|_\infty \neq 0$. So far, we did not use
  the postulated bounded acyclicity of~$F$.

  It remains to show that $\Phi_b^r$ is surjective. To this end, it
  suffices to inductively (on~$r$) establish that $\dim_\R
  \HH_b^k(T^{\times r};\R) \leq \dim_\R (\R[x_1,\dots,x_r])_k$ holds
  for all~$k \in \N$.

  The base case is handled in Corollary~\ref{cor:HbT}; moreover, the
  evaluation argument above shows that the canonical semi-norm
  on~$\HH_b^*(T;\R)$ indeed is a norm.

  For the induction step, let us assume that the claim holds for~$r-1$. 
  We recall that for group extensions~$1 \to N \to \Gamma \to Q \to 1$
  there is a Hochschild--Serre Spectral Sequence 
  \[ E_2^{pq} = \HH_b^p\bigl(Q; \HH_b^q(N;\R) \bigr)
     \Longrightarrow \HH_b^{p+q}(\Gamma;\R)
  \]
  in bounded cohomology, whenever the canonical semi-norm
  on~$\HH_b^*(N;\R)$ is a norm~\cite[Proposition~12.2.1]{monod}. 
  Applying this spectral sequence to the 
  trivial product extension
  \[ 1 \to T \to T^{\times r} \to T^{\times (r-1)} \to 1
  \]
  shows that the degree-wise dimensions of~$\HH_b^*(T^{\times r};\R)$
  are at most the degree-wise dimensions of~$\R[x_1,\dots,x_r]$
  (with~$|x_i| = 2$ for all~$i\in \{1,\dots,r\}$).  Hence, $\Phi_b^r$
  is surjective. In particular, again by the evaluation argument
  above, the canonical semi-norm on all of~$\HH_b^*(T^{\times r};\R)$ is a
  norm.
\end{proof}

\begin{rem}
  Analogous results are obtained by Monod and
  Nariman~\cite{monodnariman}, who computed the full bounded
  cohomology of the groups of orientation-preserving homeomorphisms of
  the circle and the $2$-disc.  In fact, in Proposition
  \ref{prop:circle} one can replace orbits by \emph{fat orbits} (that
  is, orbits of \emph{fat points}~\cite{monodnariman}).  This allows
  to compute the full bounded cohomology of~$\Homeo^+(S^1)$ from the
  bounded acyclicity of $\Homeo_c(\R)$.  Moreover, again using fat
  orbits, one can deduce from Proposition~\ref{prop:OF} a positive
  answer to Question~\ref{qT} for a natural countably singular
  analogue of Thompson's group $T$.
\end{rem}

\subsection{On the bounded cohomology of~$V$}

We recall the definition of~$V$:

\begin{defi}
The Thompson group~$V$ is the group of piecewise linear right-continuous
bijections~$f$ of the circle~$\R/\Z$ with the following properties:
\begin{enumerate}
\item $f$ has finitely many breakpoints, all of which lie in~$\Z[1/2]/\Z$;
\item Away from the breakpoints, $f$ is orientation-preserving and the
  slope of~$f$ is a power of~$2$;
\item $f$ preserves~$\Z[1/2]/\Z$.
\end{enumerate}
\end{defi}

It was recently proved that $V$ is acyclic~\cite{cohoV}: This had been
conjectured by Brown~\cite{cohoV2}, who already proved that $V$ is
rationally acyclic. Moreover, $V$ is uniformly perfect~\cite{thompson};
so, using the same argument as in Remark~\ref{rem:2BAc}, we deduce
that $V$ is $2$-boundedly acyclic. 

While the proof of acyclicity of $V$ is involved, the proof of
rational acyclicity is much simpler and only relies on standard
arguments in equivariant homology. A bounded analogue of equivariant
homology theory has recently been developed~\cite{kevin},
but it does not seem possible to directly translate Brown's proof 
to bounded cohomology.

A positive answer to Question~\ref{qV} would make $V$ the first
example of a non-amenable boundedly acyclic group of type~$F_\infty$.
Moreover, it would make $V$ the first tractable example of
a non-amenable boundedly acyclic finitely presented group: The only
known example~\cite{our} is very implicit.

%%%%%%%%%%%%%%%%%%%%%%%%%%%%%%%%%%%%%%%%%%%%%%%%%%%%%%%%%%%
\appendix
\section{Pseudo-mitotic groups are boundedly acyclic}\label{appendix:proof:pseudo}

We prove Theorem~\ref{thm:binatebac}. The proof is an adaption of the
original proofs for ordinary cohomology~\cite{Berrick,Varadarajan} to
the setting of bounded cohomology. In bounded cohomology, we keep
additional control on primitives, similarly to Matsumoto and Morita
for certain homeomorphism groups~\cite{MM} and similarly to the case
of mitotic groups~\cite{Loeh}.

The main ingredient in the proof is the following proposition which is
an adaptation for the pseudo-mitotic setting
of the mitotic case~\cite[Proposition~4.6]{Loeh}:

\begin{prop}\label{prop:ubc:homo}
Let $n \in \N$ and let $\kappa \in \R_{> 0}$. Let 
$$
H \xrightarrow{\varphi} H' \xrightarrow{\varphi'} K \xrightarrow{\psi} \Gamma \xrightarrow{i} P
$$
be a chain of group homomorphisms such that 
\begin{enumerate}
\item The homomorphism $i \colon \Gamma \to P$ is a pseudomitosis of
  $\Gamma$ in $P$;

\item For every $s \in \{1, \dots, n-1\}$ we have $\HH_s(\varphi') = 0$;

\item For every $s \in \{0, \dots, n-1\}$ the homomorphisms $\varphi
  $ and $\psi$ satisfy $(s, \kappa)$-$\ubc$.
\end{enumerate}
Then, for all $s \in \{1, \dots, n\}$ we have
$$
\HH_s(i \circ \psi \circ \varphi' \circ \varphi) = 0.
$$ 
Moreover, there exists a constant $c_{n, \kappa} \in \R_{>0}$ (depending
only on~$n$ and~$\kappa$) such
that the composition~$i \circ \psi \circ \varphi' \circ \varphi$
satisfies $(s, c_{n, \kappa})$-$\ubc$ for every~$s \in \{0, \cdots,
n\}$ .
\end{prop}

We give the proof of Proposition~\ref{prop:ubc:homo} in
Section~\ref{sec:proof:homo:ubc}.  Following the mitotic
case~\cite{Loeh}, we show first how to deduce
Theorem~\ref{thm:binatebac} from this result.

\begin{proof}[Proof of Theorem~\ref{thm:binatebac}]
Let $P$ be a pseudo-mitotic group. According to
Theorems~\ref{thm:MM:ubc} and ~\ref{thm:pseudo:mitotic:are:acyclic},
it is sufficient to show that for every~$n
\in \N_{>0}$ the group~$P$ satisfies~$n$-$\ubc$.

Let $n \in \N_{>0}$ and let $z \in \CC_n(P; \R)$ be a boundary, i.e.,
there exists a chain~$c \in \CC_{n+1}(P; \R)$ such
that~$\partial_{n+1} c = z$. Since both $z$ and $c$ involve only
finitely many elements of~$P$, there exists a finitely generated
subgroup~$\Gamma_0$ of~$P$ such that
$
z \in \CC_{n}(\Gamma_0; \R)$
and
$c \in \CC_{n+1}(\Gamma_0; \R).
$
We show that the inclusion~$\Gamma_0 \hookrightarrow P$
satisfies $(n, \kappa_n)$-$\ubc$, where $\kappa_n \in \R_{>0}$
only depends on~$n$. This condition readily implies that $P$
satisfies $n$-$\ubc$, whence the thesis.

Since $P$ is pseudo-mitotic, $\Gamma_0$ has a mitotis into~$P$. Let
$\psi_0, \psi_1 \colon \Gamma_0 \to P$ and $g \in P$ be witnesses of
such a pseudo-mitosis. Then, $\Gamma_0$ also admits a pseudo-mitosis
into the following finitely generated subgroup of~$P$:
$$
\Gamma_1 = \langle \Gamma_0, \psi_0(\Gamma_0), \psi_1(\Gamma_0), g\rangle \subset P.
$$
By iterating this construction we get a sequence~$\Gamma_0 \leq
\Gamma_1 \leq \cdots \leq P$ of finitely generated groups such that at
each step the inclusion~$\Gamma_{j} \hookrightarrow \Gamma_{j+1}$ is a
pseudo-mitosis.

Following \emph{verbatim} the proof of the mitotic
case~\cite[Theorem~1.2]{Loeh}, by induction on~$n \geq 1$ and using
Proposition~\ref{prop:ubc:homo}, one can show that the inclusion
of~$\Gamma$ into a sufficiently large $\Gamma_{j_n}$ satisfies $(n,
\kappa_n)$-$\ubc$, where $\kappa_n$ only depends on~$n$.

Using the fact that $\Gamma_{j_n} \leq P$, this implies that there
exists $c' \in \CC_{n+1}(P; \R)$ with
$$
\partial_{n+1} c' = z \quand |c'|_1 \leq \kappa_n \cdot |z|_1.
$$
This shows that $P$ satisfies $n$-$\ubc$ for all positive degrees~$n$; 
whence, $P$ is boundedly acyclic (Theorem~\ref{thm:MM:ubc}).
\end{proof}

%%%%%%%%%%
\subsection{Proof of Proposition~\ref{prop:ubc:homo}}\label{sec:proof:homo:ubc}

This section is devoted to the proof of
Proposition~\ref{prop:ubc:homo}.  The proof is based on a refinement
of Varadarajan's proof~\cite[Proposition~1.4]{Varadarajan},
additionally taking the norm of the morphisms involved into account.
Our approach will closely follow the mitotic
case~\cite[Appendix~A]{Loeh}.

\begin{proof}[Proof of Proposition~\ref{prop:ubc:homo}]
We prove the statement in degree $n \in \N$. For convenience we write 
$$
f \coloneqq \psi \circ \varphi' \circ \varphi.
$$
Since $\HH_s(\varphi') = 0$ for every $s \in \{1, \dots, n-1\}$,
also $\HH_s(f) = 0$ in the same degrees.
The fact that
$\HH_s(i \circ f) = 0$ 
for every~$s \in \, \{1, \dots, n\}$ was already proved by
Varadarajan~\cite[Proposition~1.4]{Varadarajan}. 
In order to
adapt the mitosis proof~\cite[Appendix~A]{Loeh} to the case 
of pseudo-mitoses, it is convenient to recall Varadarajan's argument.

Let $\psi_0, \psi_1 \colon \Gamma \to P$ and $g \in \, P$ be witnesses
of the pseudo-mitosis $i \colon \Gamma \to P$.  Then, we define the map
\begin{align*}
\mu \colon \Gamma \times \Gamma &\to P \\
(g', g) &\mapsto g' \cdot \psi_1(g) .
\end{align*}
Notice that $\mu$ is a group homomorphism by Condition~(2) of the
definition of pseudo-mitosis~\cite[proof of Proposition~1.4]{Varadarajan}.

Let $j_1 \colon H \to H \times H$ and $j_2 \colon H \to H \times H$ be
the inclusions into the first and the second factor, respectively.
Let $\gamma_g$ denote the conjugation with respect to $g$, i.e.,
$\gamma_g(g') = g g' g^{-1}$ for every $g' \in P$. Moreover, let
$\Delta_H \colon H \to H \times H$ denote the diagonal homomorphism.
Then, for every $h \in H$, we have:
\begin{align*}
\gamma_g \circ \psi_1 \circ f (h) = \psi_0 \circ f (h) &= \psi_0\bigl(f(h)\bigr) \\
&= f(h) \cdot \psi_1\bigl(f(h)\bigr) \\
&= \mu\bigl(f(h), f(h)\bigr) \\
&= \mu \circ (f \times f) \circ \Delta_H (h).
\end{align*}
On the other hand, we also have 
\begin{align*}
  \mu \circ (f \times f) \circ j_1
  & = \mu (f, 1) = f = i \circ f 
  \quad\text{and}
  \\
  \mu \circ (f \times f) \circ j_2
  & = \mu(1, f) = \psi_1(f) = \psi_1 \circ f.
\end{align*}
Hence, the K\"unneth Formula (and its naturality) together with the
assumption that $\HH_s(f) = 0$ for all $s \in \{1, \dots, n-1\}$ imply
that the following diagram commutes (similarly to the
mitotic case~\cite[p.~729]{Loeh}):
$$
\xymatrix{
  \HH_n(H \times H; \R) \ar[rr]^-{\HH_n(p_1) \oplus \HH_n(p_2)} \ar[d]_-{\HH_n(f \times f)}
  && \HH_n(H; \R) \oplus \HH_n(H; \R) \ar[d]^-{\HH_n(f) \oplus \HH_n(f)} \\
  \HH_n(\Gamma \times \Gamma; \R) \ar[d]_-{\HH_n(\mu)}
  && \HH_n(\Gamma; \R) \oplus \HH_n(\Gamma; \R) \ar[ll]^-{\HH_n(i_1) + \HH_n(i_2)} \ar[d]^-{\HH_n(i) \oplus \HH_n(\psi_1)} \\
\HH_n(P; \R) && \HH_n(P; \R) \oplus \HH_n(P; \R) \ar[ll]^-{\id + \id},
}
$$
where $p_1, p_2 \colon H \times H \to H$ denote the projections onto the two
factors and $i_1, i_2 \colon \Gamma \to \Gamma \times \Gamma$ the
inclusions of the factors.  The commutativity of the previous diagram
leads to
\begin{align*}
\HH_n(\gamma_g) \circ \HH_n (\psi_1 \circ f) &= \HH_n (\mu \circ (f \times f) \circ \Delta_H) \\
&= \HH_n (i \circ f) + \HH_n(\psi_1 \circ f).
\end{align*}
Since the conjugation~$\gamma_{g}$ is trivial in homology, i.e., $\HH_n(\gamma_g) =
\id$, we obtain
$$
\HH_n(i \circ f) = 0.
$$
We are thus reduced to show that the previous construction can be
controlled in such a way that $i \circ f$ satisfies the required
$\ubc$ condition.

Let $z \in \partial_{n+1}(\CC_{n+1}(H; \R))$. We will construct
a controlled $\partial_{n+1}$-primitive for~$\CC_n(i \circ f)$.
Following the mitotic case~\cite[p.~731]{Loeh}, we have
\begin{equation}\label{eq:f:times:f:E}
(f \times f)_* \circ \Delta_{H*} (z) = (f \times f)_* \circ j_{1*} +
  (f \times f)_* \circ j_{2*} + \partial_{n+1} E'(z),
\end{equation}
on the chain level, where $E'$ is bounded and $\|E'\|$ admits a bound
that only depends on the given~$\kappa \in \R_{>0}$ and~$n$ (the proof
uses hypothesis~(3) in the statement).

To complete the construction of a controlled
$\partial_{n+1}$-primitive for~$\CC_n(i \circ f)$ we consider the
following chain homotopy
\begin{align*}
\Theta_n \colon \CC_n(P; \R) &\to \CC_{n+1}(P; \R)  \\
(g_1, \dots, g_n)
&\mapsto \sum_{j = 1}^{n+1} (-1)^{j+1} \cdot (g_1, \dots, g_{j-1}, g, g^{-1} g_j g, \dots, g^{-1} g_n g)
\end{align*}
between $\CC_*(\gamma_g)$ and the identity. That the previous map is
in fact such a chain homotopy and that $\|\Theta_n\| \leq n+1$ is
proved as in the mitotic case~\cite[Lemma~A.2]{Loeh} (notice that for
convenience we changed the sign of~$\Theta$). 
We then have:
\begin{align*}
  (i \circ f)_*(z)
  &= \bigl(\mu \circ (f \times f) \circ j_1\bigr)_* (z) \\
  &= \bigl(\mu \circ (f \times f) \circ \Delta_H\bigr)_*(z)
  \\
  & \quad - \bigl(\mu \circ (f \times f) \circ j_2\bigr)_*(z) - \mu_* \circ \partial_{n+1} \circ E'(z) \\
  &= (\gamma_g)_* \circ (\psi_1 \circ f)_* (z)
  \\
  & \quad - (\psi_1 \circ f)_*(z) - \partial_{n+1} \circ \mu_* \circ E'(z) \\
&= (\partial_{n+1} \circ  \Theta + \Theta \circ \partial_n) \circ (\psi_1 \circ f)_*(z) - \partial_{n+1} \circ \mu_* \circ E'(z) \\
&= \partial_{n+1} \bigl(\Theta \circ (\psi_1 \circ f)_*(z) - \mu_* \circ E'(z)\bigr),
\end{align*}
where we moved from the first line to the second one using the
formula~\eqref{eq:f:times:f:E} and the last equality holds because $z$
is a cycle. Moreover, using the fact that group homomorphisms induce
norm non-increasing chain maps, we have that the norm
$$
\bigl\|\Theta \circ (\psi_1 \circ f)_*(z) - \mu_* \circ E'(z)\bigr\|
\leq \| \Theta \| + \|E'\|
\leq n+1 + \| E'\|
$$
is bounded from above by a quantity~$c_{n, \kappa} \in \R_{>0}$
depending only on $n$ and~$\kappa$ (since this is true for $E'$ and
$\Theta$).  This shows that $i \circ f$ satisfies $(n, c_{n,
  \kappa})$-$\ubc$, as claimed.
\end{proof}

%%%%%%%%%%%%%%%%%%%%%%%%%%%%%%%%%%%%%%%%%%%%%%%%%%%%%%%%%%
\section{Computation of cup-powers}\label{appx:eupowers}

In the following, we give the combinatorial part of the proof of
Claim~\ref{claim:eucup} in the proof of
Proposition~\ref{prop:circle}. We use the notation established in the
proof of Proposition~\ref{prop:circle}.

\begin{lemma}
  For all~$k \in \{0,\dots, n/2\}$, we have
  \[ \alt(f_2{}^{\cup k}) = \frac{2^k \cdot k!}{(2k)!} \cdot f_{2k}.
  \]
\end{lemma}
\begin{proof}
  Proceeding inductively, it suffices to show that
  \[ \alt(f_2 \cup f_{2(k-1)}) = \frac1{2k-1} \cdot f_{2k}.
  \]
  This is a purely combinatorial statement. Let $(s_0, \dots, s_{2k})$
  be a circularly ordered $(2k+1)$-tuple over~$S$. In order to
  simplify notation, we will write~$f(i_0,\dots, i_p)$ for~$f(s_{i_0},
  \dots, s_{i_p})$, etc.  In this notation, since $(s_0, \dots,
  s_{2k})$ is circularly ordered, it suffices to show that $A :=
  \alt(f_2 \cup f_{2(k-1)}) (0,\dots, 2k)$ is equal to
  \[ \frac1{2k-1}.
  \]
  By definition of~$\alt$ and the cup-product on simplicial cochains,
  we have 
  \begin{align*}
    A 
  = \frac1{(2k +1)!} \cdot
  \sum_{\sigma \in \Sigma(0,\dots, 2k)} \sgn(\sigma)
  & \cdot f_2\bigl(\sigma(0),\sigma(1),\sigma(2)\bigr)
  \\[-.5em]
  & \cdot f_{2(k-1)} \bigl(\sigma(2), \dots, \sigma(2k) \bigr).
  \end{align*}
  Because $2k$ is even, we can fix one position and obtain via cyclic
  permutations that
  \begin{align*}
    A 
  = \frac{2k+1}{(2k +1)!} \cdot
  \sum_{\sigma \in \Sigma(0,\dots, 2k-1)} \sgn(\sigma)
  & \cdot f_2\bigl(\sigma(0),\sigma(1),2k\bigr)
  \\[-.5em]
  & \cdot f_{2(k-1)} \bigl(2k,\sigma(2), \dots, \sigma(2k-1) \bigr).
  \end{align*}
  Flipping $\sigma(0)$ and $\sigma(1)$ changes both the sign
  of~$\sigma$ and of~$f_2(\sigma(0), \sigma(1),2k)$. Therefore,
  we obtain
  \begin{align*}
    A
  & = \frac{2}{(2k)!} \cdot
  \sum_{i=0}^{2k-2}\sum_{j=i+1}^{2k-1}
  \sum_{\sigma \in \Sigma(X_{i,j})} \sgn([i,j]*\sigma)
  \cdot f_2\bigl(i,j,2k\bigr)
  \cdot f_{2(k-1)} \bigl([2k]*\sigma \bigr)
  \\
  & = \frac{2}{(2k)!} \cdot
  \sum_{i=0}^{2k-2}\sum_{j=i+1}^{2k-1}
  \sum_{\sigma \in \Sigma(X_{i,j})} \sgn([i,j]*\sigma)
  \cdot 1 \cdot \sgn(\sigma).
  \end{align*}
  Here, we use the following notation: $X_{i,j} := \{0,\dots,2k-1\}
  \setminus \{i,j\}$; the permutation/tuple~$[i,j]*\sigma$
  on~$\{0,\dots,2k-1\}$ is obtained by using $i$, $j$ in the first two
  positions and then filling up with~$\sigma$, etc.

  Let $[X_{i,j}]$ be the sequence of elements in~$X_{i,j}$,
  in order. Then
  \begin{align*}
    \sgn(\sigma) \cdot \sgn\bigl([i,j]*\sigma\bigr)
    = \sgn\bigl([i,j] * [X_{i,j}]\bigr)
    = \sgn (j-1) \cdot \sgn(i)
  \end{align*}
  for all~$i \in \{0,\dots,2k-2\}$, $j \in \{i+1,\dots,2k-1\}$;
  here, we set~$\sgn(x)$ of~$x \in \N$ to~$+1$ if $x$ is even,  and
  to~$-1$ if $x$ is odd.
  We distinguish two cases:
  \begin{itemize}
  \item
    If $i$ is odd, then $\sum_{j=i+1}^{2k-1} \sgn(j-1)$ is zero,
    because there are equally many even and odd numbers
    in~$\{i+1, \dots, 2k-1\}$.
  \item
    If $i$ is even, then $\{i+1,\dots, 2k-1\}$ contains one
    more odd number than even numbers, whence
    \[ \sgn(i) \cdot \sum_{j=i+1}^{2k-1} \sgn(j-1)
    = 1.
    \]
  \end{itemize}
  Therefore, since there are $k$ even numbers inside $\{0, \cdots, 2k -2\}$, we obtain 
  \begin{align*}
    A
  & = \frac{2}{(2k)!} \cdot
      k \cdot |X_{i,j}| ! 
    = \frac{2}{(2k)!} \cdot k \cdot (2k-2)!
    = \frac{1}{2k-1}, 
  \end{align*}
  as claimed.
\end{proof}

Similar computations can also be found in the literature~\cite{jekel}.

%%%%%%%%%%%%%%%%%%%%%%%%%%%%%%%%%%%%%%%%%%%%%%%%%%%%%%%%%%%
% bib
 
\bibliographystyle{abbrv}
\bibliography{bacbib}

\end{document}